\documentclass[11pt]{article}
\newcommand{\sfrac}[2]{{\textstyle\frac{#1}{#2}}}
\usepackage{amsmath,amssymb}
\usepackage{graphicx}  
\usepackage{tikz}
\newcommand{\Ex}{\mathbb{E}}
\newcommand{\CC}{\mathbb{C}}
 \renewcommand{\Pr}{{\mathbb{P}}}

\newcommand{\PP}{\mathcal P}

\newcommand{\Reals}{{\mathbb{R}}}
\newcommand{\bX}{\mathbf{X}}

\newcommand{\bx}{\mathbf{x}}
\newcommand{\by}{\mathbf{y}}
 \newcommand{\eps}{\varepsilon}
\newtheorem{Lemma}{Lemma}
\newtheorem{Theorem}[Lemma]{Theorem}
\newtheorem{Proposition}[Lemma]{Proposition}

\newtheorem{Conjecture}[Lemma]{Conjecture}

 \newcommand{\qed}{\ \ \rule{1ex}{1ex}}
 
 \newcommand{\Bin}{\mathrm{Bin}}
 \newcommand{\ball}{\mathrm{ball}}
 \newcommand{\dist}{\mathrm{dist}}
 \newcommand{\support}{\mathrm{support}}
 
\begin{document}

\title{Markov chains and mappings of distributions on compact spaces II: Numerics and Conjectures}
 \author{David J. Aldous\thanks{Department of Statistics,
 367 Evans Hall \#\  3860,
 U.C. Berkeley CA 94720;  aldous@stat.berkeley.edu;
  www.stat.berkeley.edu/users/aldous.}
  \and 
  Madelyn Cruz\thanks{Department of Mathematics, University of Michigan; 
  mccruz@umich.edu}
    \and
  Shi Feng\thanks{Department of Mathematics, Cornell University;
  sf599@cornell.edu}
}

 \maketitle
 
 \begin{abstract}
 Consider a compact metric space $S$ and a pair $(j,k)$ with $k \ge 2$ and $1 \le j \le k$.
 For any probability distribution $\theta \in \PP(S)$, define a Markov chain on $S$ by:
 from state $s$, take $k$ i.i.d. ($\theta$) samples, and jump to the $j$'th closest. 
 Such a chain converges in distribution to a unique stationary distribution, say $\pi_{j,k}(\theta)$.
This defines a mapping $\pi_{j,k}: \PP(S) \to \PP(S)$.
What happens when we iterate this mapping?
In particular, what are the fixed points of this mapping? 
A few results are proved in a companion article \cite{mappings_short};  this article, not intended for formal publication,  records
numerical studies and conjectures.
 \end{abstract}

This document records our investigation into the novel topic described below.
Numerics and heuristics suggest a variety of general conjectures,
but what we can actually prove is seriously limited.
The rigorous results stated here (Theorems \ref{T:1}, \ref{P:k=2}, \ref{Conj}, \ref{P:0122}) are proved in an accompanying article \cite{mappings_short}.
This document is intended as informal discussion of a big picture. 
The introduction below expands on the introduction to \cite{mappings_short}.

\section{Introduction}
We start with informal background that led to the formulation of the questions to be studied.
If I suggest a specific compact metric space $(S,d)$ and ask you to define some ``interesting" mapping $f: S \to S$ then you can certainly do so.
However, if I challenge you to define a scheme which can be used in any ``generic" such space $S$, without exploiting any structure of the specific space, then that seems very difficult.
For instance you could define 
\[ f(s) := \arg \max_y d(s,y) \]
that is the most distant point from $s$; this works  for any space $S$ with the property that the most distant point is always unique.
Are there any more interesting general schemes?

Instead let us write $\PP(S)$ for the space of probability distributions on $S$. 
Consider the analogous challenge: define a scheme which can be used to specify a mapping $\PP(S) \to \PP(S)$ for a generic $S$.
In contrast to the case of mappings $S \to S$, there are perhaps many ways to do so.
In particular, there are schemes involving Markov chains.
For instance, given $\theta \in \PP(S)$ you can define a transition kernel $K(s,\cdot)$ for $s \in \support(\theta)$ by conditioning $\theta$
to the ball of radius $1$ around $s$:
\[ K(s,A) = \theta(A \cap \ball(s,1) )/\theta (\ball(s,1)) . 
\]
Under a simple connectivity condition\footnote{Essentially, that $\cup_{s \in \support(\theta)}  \ball (s,1)$ is connected.}  on $\theta$, the corresponding Markov chain has a unique stationary distribution, which we can call $\pi(\theta)$.
So this scheme defines a map $\pi:\PP_0(S) \to \PP(S)$ for every $S$ but only for the subset $\PP_0(S) \subseteq \PP(S)$ satisfying the connectivity condition. 
One could view this as a continuous-space analog of random walk on a graph.

In this project we study a scheme which works for {\bf all} compact metric spaces $S$ and {\bf all} of $\PP(S)$. 
Given $\theta \in \PP(S)$, one can consider the Markov chain that from 
point $s$ makes the step  to the nearer of $2$ random points drawn i.i.d. from $\theta$, breaking ties uniformly at random.
This scheme naturally generalizes as follows: fix $k \ge 2$ and $1 \le j \le k$, 
and step from $s$ to the $j$'th nearest of  $k$ random points drawn i.i.d. from $\theta$, again breaking ties uniformly at random. 
This defines a kernel $K^{\theta,j,k}$ on $S$.
Write the associated chain as $\bX^{\theta,j,k} = (X^{\theta,j,k}(t), t = 0, 1, 2, \ldots)$.
Theorem  \ref{T:1}  proves that this chain always has a unique stationary distribution,  which we can call $\pi_{j,k}(\theta)$.
So now we have defined a mapping $\pi_{j,k}: \PP(S) \to \PP(S)$ for every $S$.
Moreover Theorem  \ref{T:1}   proves that the distributions $\theta$ and $\pi_{j,k}(\theta)$ are mutually absolutely continuous, so in particular
have the same support.

These maps $\pi_{j,k}$ have apparently not been studied previously, even for special spaces $S$, so 
the purpose of this project is to initiate their study.
There is a range of problems one might consider.
The proof of Theorem  \ref{T:1}, using the natural coupling, does not yield any helpful explicit expression for the stationary distribution.
So one can ask
\begin{quote}
On a given space $S$, is there any informative description of  $\pi_{j,k}(\theta)$ in terms of $\theta$?
\end{quote} 
In this project our focus is different.
Given a mapping $\pi$ from a space to itself, it is mathematically natural to consider iterates
\[ \pi^{n+1}(\theta) = \pi(\pi^n(\theta)), n \ge 1. \]
In our setting it seems plausible that (at least for typical initial $\theta$) the iterates should converge to some fixed point, that is we expect
\begin{equation}
 \pi^n_{j,k}(\theta) \to_w \phi \mbox{ as } n \to \infty
 \label{pin}
 \end{equation}
and then we expect the limit to satisfy 
\begin{equation}
\pi_{j,k}(\phi) = \phi . 
\label{phiFP}
\end{equation}
We will call such a construction $(\pi^n_{j,k}(\theta), n \ge 0)$ the {\em iterative procedure}.
Note that this does not have any simple stochastic process interpretation, in contrast to the mapping $\theta \to \pi_{j,k}(\theta)$ derived from the Markov chain.

If $\phi$ satisfies (\ref{phiFP}) we call it a {\em fixed point} or an {\em invariant distribution} for $\pi_{j,k}$.
One can view this property as a kind of ``self-similarity under sampling" property.
The original motivation for this project was the hope that a typical compact metric space might have ``interesting" invariant distributions.
Our first attempt at finding such distributions was via numerically implementing the iterative procedure.
But in doing so, with initial distributions lacking symmetry, we observed that typically there is a limit but it is supported on only one or two points.
This seemed counter-intuitive, and prompted the further study of fixed points in this article.

Let us emphasize the trivial observation:

\medskip \noindent
{\bf On any $S$ and for any $(j,k)$, two types of measures are always invariant:
we call these the {\em omnipresent} measures.}
\begin{itemize}
\item The distribution $\delta_s$ degenerate at one point $s$;
\item The uniform two-point distribution 
$\delta_{s_1,s_2} = \frac{1}{2}(\delta_{s_1} + \delta_{s_2})$.
\end{itemize}

\medskip \noindent
Are there others?

\subsection{Formulating a program}
There are some subtleties involved in formulating a precise program, as follows.

(i) Convergence in (\ref{pin}) is weak convergence.
Recall that a Markov chain is a {\em Feller} chain if its kernel $K$ is weakly continuous:
\begin{equation}
  K(s_n,\cdot) \to_w K(s,\cdot) \mbox{ for all $s_n \to s$}.
  \label{def:Feller}
  \end{equation}
It is straightforward and well known (\cite{douc} section 12.3) that a Feller chain on a compact $S$ always has at least one stationary distribution.
Our chains  $\bX^{\theta,j,k}$ are not necessarily Feller, the obstacle arising when the chain involves ties -- see discussion in section \ref{sec:Feller}.
But it turns out that we do not need the Feller property for the foundational result, Theorem \ref{T:1}.
On the other hand one
might hope that the mapping  $\pi_{j,k}: \PP(S) \to \PP(S)$ is continuous, but we can only prove that it is continuous at (roughly speaking) distributions $\theta$ for which the Feller property holds -- 
see Proposition \ref{P:Feller}.
So convergence (\ref{pin}) does not automatically imply that the limit is a fixed point (\ref{phiFP}), though we do not know a counter-example.

(ii)
For given $\theta$ the iterates $\pi^n_{j,k}(\theta), n \ge 0$ all have the same support but we will see that the limit usually
has much smaller support,
%\footnote{In this sense our setting is rather different from that of classical dynamical systems and ergodic theory.} 
so we cannot insist on limits having full support on $S$.  
On the other hand, any fixed point distribution over a space $S_0$ is automatically a fixed point distribution over a superspace $S_1 \supset S_0$, and to talk about 
``fixed points for $S_1$" we want the fixed point to have some relation with $S_1$.
So an appropriate precise problem formulation is

\begin{quote}
{\bf Program.}   (a) Given a compact metric space $(S,d)$ and given $(j,k)$, find all the fixed points of $\pi_{j,k}$ 
with full support on $S$. \\
(b) Given a compact metric space $(S,d)$ and given $(j,k)$, determine which invariant measures (not necessarily with full support
on $S$) 
arise as limits 
(\ref{pin}) of the iterative procedure from some initial $\theta$ with full support on $S$.
\end{quote}

\subsection{Toward a big picture}

This article reports on our extensive study via numerics of different various spaces $S$,
which provides the following heuristic ``big picture" of issues in the program above.

\medskip

(a) For $k = 2$, there are no invariant measures other than the omnipresent ones, except perhaps for
``exist by symmetry" ones; 
with that exception, for $j = 1, k = 2$ the iterates \eqref{pin} converge to some $\delta_s$,  
and for $j = 2, k = 2$ the iterates \eqref{pin} converge to some $\delta_{s_1,s_2}$. 
The precise limits $(s, s_1, s_2)$ may depend on the initial $\theta$.  In the case of 
$\delta_{s_1,s_2}$, the pair $(s_1,s_2)$ is a local maximum of $d(\cdot, \cdot)$.

(b) For larger $k$, for some types of space $S$ there are additional {\em sporadic} invariant measures; 
we don't see a pattern.

(c) For large $k$, as $j$ increases we see a transition, around $j/k = 0.7$, between convergence to some $\delta_s$ 
and convergence to some $\delta_{s_1,s_2}$.
However there seems no reason to believe that there is a universal value near $0.7$.

(d) Except for the omnipresent ones, all invariant measures $\phi$ that we have encountered 
are {\em unstable}, in that from any initial distribution of the form $\phi$ plus a generic (not symmetry-preserving) 
small perturbation, the iterates converge to some $\delta_s$ or  $\delta_{s_1,s_2}$.

We have succeeded in proving only a few small parts on this picture. 
The shorter article \cite{mappings_short} contains proofs of the following theorems, only stated in this document.

\begin{itemize}
\item Theorem \ref{T:1} is the Markov chain convergence result stated above.
%\item Results in section \ref{sec:2or3} for $|S| = 2$ or $3$  are consistent with general picture  above.
\item Theorem \ref{P:k=2}: 
{\em For every  $S$, the set of invariant distributions for $\pi_{1,2}$ is the same as the set of invariant distributions for $\pi_{2,2}$.}  
This is surprising, in that  apparently (as in  (a) above) the iterates almost always converge to some $\delta_s$ for $\pi_{1,2}$,
but to some $\delta_{s_1,s_2}$ for $\pi_{2,2}$.
 \item Theorem \ref{Conj}:
  {\em There are no $\pi_{1,2}$ or $\pi_{2,2}$-invariant distributions on 
the space of finite binary tree leaves 
(see section \ref{sec:tree})
other than the omnipresent ones.}
\item Theorem \ref{P:0122}:
 {\em There are no $\pi_{1,2}$ or $\pi_{2,2}$-invariant distributions on 
 the interval $[0,1]$ other than the omnipresent ones.}
\end{itemize}

\subsection{Outline of this article}
Our main focus is on simulation results, leading to formulation of conjectures.  
We also provide technical comments.

\begin{itemize}
\item  Section \ref{sec:2}  serves to define the map $\pi_{j,k}$ (Theorem \ref{T:1}) and discusses the Feller property.
\item Section \ref{sec:symm} discusses the ``symmetry is preserved" feature of $\pi_{j,k}$, implying that some invariant distributions exist ``by symmetry", but also suggesting that
initial distributions with symmetry are atypical.
It also gives the counter-intuitive, though easily verified, result (Theorem \ref{P:k=2}) that the invariant distributions for $\pi_{2,2}$ coincide with 
the invariant distributions for $\pi_{1,2}$. 
\item Section \ref{sec:binomial} observes that even the 2-point case $S = \{a,b\}$ is interesting: as well as the obvious fixed points, for certain values of $(j,k)$ there are extra fixed points, but these 
appear (from numerics) to be {\em unstable} under the mapping $\pi_{j,k}$.
\item Section \ref{sec:finite} gives examples of non-uniform invariant distributions for finite spaces $S$.
We record the strong conjecture that, for finite $S$,  the iterative process from {\em almost all} initial distributions converges to 
one of the omnipresent limits; 
this is tantamount to saying that other fixed points are unstable. 
\item Section \ref{sec:tree} describes the {\em binary tree leaves} space, on which we can prove (Theorem \ref{Conj}) that for $k=2$ the 
only invariant distributions are the omnipresent ones.
\item The natural basic example of a continuous compact space is the unit interval, and we give a simulation study in section \ref{sec:interval}.
In examples we see convergence to limit distributions  supported on one or two points; we do not know whether (for any $(j,k)$) there exist invariant distributions with full support, though
we have partial results including Theorem \ref{P:0122}.
\item Section  \ref{sec:circle} studies the case of the circle, for which the uniform distribution is invariant.
Section \ref{sec:highdim} considers a high-dimensional case.  
Both are simulation studies.
\end{itemize}

\section{Existence and uniqueness of stationary distributions}
\label{sec:2}

\begin{Theorem}[\cite{mappings_short} Theorem 1]
\label{T:1}
Consider a compact metric space $(S,d)$ and a probability distribution $\theta \in \PP(S)$.
For each pair $1 \le j \le k, \ k \ge 2$, the Markov chain
$\bX^{\theta,j,k} = (X^{\theta,j,k}(t), t = 0, 1, 2, \ldots)$ 
has a unique stationary distribution $\pi_{j,k}(\theta)$.
From any initial point, the variation distance $D(t)$ between $\pi_{j,k}(\theta)$ and the distribution of $X^{\theta,j,k}(t)$ satisfies
\begin{equation}
 D(2t) \le (1 - 1/k^{k-1})^t, \quad 1 \le t < \infty 
 \label{VD}
 \end{equation}
and so there is convergence to stationarity in variation distance. 
Moreover, for $\pi = \pi_{j,k}(\theta)$ 
\begin{equation}
\theta^k(A) \le \pi(A) \le k \theta(A) , \ A \subseteq S 
\label{tkA}
\end{equation}
and so $\pi$ and $\theta$ are mutually absolutely continuous.
\end{Theorem}

{\bf Remarks.}
(a) Note that the bound on variation distance depends only on $k$.

(b) The variation distance bound (\ref{VD}) is exponentially decreasing in time, but it is more natural to consider {\em mixing time} in the sense of  \cite{MCMC}. 
The example of the uniform distribution $\theta$ on a  2-point space with $j = 1$ shows that the mixing time as a function of $k$ 
can be order $2^k$.

(c) As mentioned earlier, the proof of Theorem \ref{T:1} does not say anything about $\pi_{j,k}(\theta)$ except (\ref{tkA}).
We do not know if there are  informative analytic descriptions of $\pi_{j,k}(\theta)$  in terms of $\theta$.

\subsection{Concerning the Feller property}
\label{sec:Feller}
As an example to show that the Markov chains $\bX^{\theta,j,k}$ are not necessarily Feller,
consider $S = [0,1]$ and $\theta = \frac{1}{3}(\delta_0 + \delta_1 + \lambda)$
for Lebesgue measure $\lambda$.
Here the requirement (\ref{def:Feller}) that $K(x_i,\cdot) \to_w K(\frac{1}{2}, \cdot)$ as $x_i \to \frac{1}{2}$ fails.

This difficulty occurs because there are ties in the distances to the points 
$(Y_i, 1 \le i \le k)$ sampled i.i.d. from $\theta$.
Suppose, for given  $\theta \in \PP(S)$ 
and given $s \in S$,
\begin{equation}
d(s,Y_i) \mbox{ has non-atomic distribution} .
\label{Fel:1}
\end{equation}
This of course cannot hold for finite $S$.
Property (\ref{Fel:1}) implies
\begin{equation}
\Pr( d(s,Y_i) = d(s,Y_j)) = 0 \mbox{ for } j \neq i.
\label{Fel:2}
\end{equation}
It is straightforward to show that, if (\ref{Fel:2}) holds for all $s \in S$, 
then the Markov chain $\bX^{\theta,j,k}$ is Feller.
Let us consider the slightly weaker property: for given $\theta \in \PP(S)$
\begin{equation}
d(s,Y_i) \mbox{ has non-atomic distribution for $\theta$-almost all } s \in S.
\label{Fel:3}
\end{equation}
Note this is equivalent to
\begin{equation}
\Pr(d(Y_1,Y_2) = d(Y_1,Y_3)) = 0 .
\label{Fel:3a}
\end{equation}
We can now give a partial continuity property for the maps $\pi_{j,k}$.
\begin{Proposition}
\label{P:Feller}
If $\theta \in \PP(S)$ has property (\ref{Fel:3}) then
$\pi_{j,k}(\theta_n) \to_w \pi_{j,k}(\theta)$
for all $\theta_n \to_w \theta$.
\end{Proposition}
\begin{proof}
Consider  $\theta_n \to_w \theta$.
Write $(X(t), t =0, 1, 2, \ldots)$ for  the chain $\bX^{\theta,j,k}$ started with distribution $\theta$ 
and write $(X^n(t), t =0, 1, 2, \ldots)$ for  the chain $\bX^{\theta_n,j,k}$ started with distribution $\theta_n$.
It is sufficient to prove that, for fixed $t$, we have convergence in distribution
\begin{equation}
 \dist( X^n(t))  \to_w \dist( X(t))  \mbox{ as } n \to \infty 
 \label{Fel:4}
 \end{equation}
because then the result will follow from (\ref{VD}).
Proving this would make a nice exercise in a course on weak convergence.
We will write out the case $(j,k) = (1,2)$: the general case is similar.

We prove (\ref{Fel:4}) by induction on $t$.  
Fix $t$, assume (\ref{Fel:4}) for that $t$, and recall that by construction
\begin{equation}
\dist( X(t)) \le k \theta .
\label{ktheta}
\end{equation}
There exist couplings $(Y_i,Y	_i^n)$ of $\theta$ and $\theta_n$ such that
\[
a(n) := \Ex d(Y_i,Y_i^n) \to 0 \mbox{ as } n \to \infty
\]
and there exist couplings $(X(t), X^n(t))$ of $\phi$ and $\phi_n$ such that
\[
b(n) := \Ex d(X(t),X^n(t)) \to 0 \mbox{ as } n \to \infty .
\]
We have
\begin{eqnarray*}
X(t+1) &= & Y_1 \mbox{ if } d(X(t), Y_1) < d(X(t), Y_2) \\
	  &= & Y_2 \mbox{ if } d(X(t), Y_1) > d(X(t), Y_2) 
 \end{eqnarray*}
 and also the case of ties.
 Combining with the analog for $X^n(t+1)$ we have
\begin{eqnarray*}
d(X(t+1),X^n(t+1)) &=& d(Y_1, Y^n_i) \mbox{ on } A_n \\
		A_n& :=& \{ d(X(t), Y_1) < d(X(t), Y_2) \mbox{ and } d(X^n(t), Y^n_1) < d(X^n(t), Y^n_2) \} \\
d(X(t+1),X^n(t+1)) &=& d(Y_2, Y^n_2) \mbox{ on } B_n \\
		B_n &:=& \{ d(X(t), Y_1) > d(X(t), Y_2) \mbox{ and } d(X^n(t), Y^n_1) > d(X^n(t), Y^n_2) \} 
\end{eqnarray*}
Now $A_n$ contains the event $A_n(\eps)$, where for $\eps > 0$
\[
 A_n(\eps) :=  \{ d(X(t), Y_1) < d(X(t), Y_2) - 3 \eps; \ d(X(t),X^n(t)) \le \eps; \ d(Y_2,Y^n_2) \le \eps \}
\]
and $B_n$ contains an analogous event $B_n(\eps)$.
Taking account of the remaining cases, we obtain
\[
\Ex d(X(t+1),X^n(t+1)) \le 2 a(n) + \Delta \Pr( (A_n(\eps)  \cup B_n(\eps))^c)
\]
where $\Delta < \infty$ is the diameter of $S$.
Using Markov's inequality to bound the probability of events like 
$\{ d(X(t),X^n(t)) > \eps\}$,
we obtain
\[
\Ex d(X(t+1),X^n(t+1)) \le 2 a(n) + \Delta \left( \sfrac{2(a(n)+b(n))}{\eps} + \Pr(  \vert d(X(t), Y_1) -  d(X(t), Y_2) \vert \le 3 \eps      \right) .
\]
Letting $n \to \infty$ we see that it suffices to prove
\[  \Pr \left(  \vert d(X(t), Y_1) -  d(X(t), Y_2) \vert \le 3 \eps      \right) \to 0 \mbox{ as } \eps \to 0 \]
but this follows from (\ref{Fel:2}) and (\ref{Fel:4}).
\qed
\end{proof}

\section{General remarks about fixed points of the mapping}
\label{sec:symm}
On every compact metric space $S$ we have an obvious ``preservation of symmetry" result
\begin{Lemma}
\label{L:sym}
If $\theta \in \PP(S)$ is invariant under an isometry $\iota$ of $S$ then $\pi_{j,k}(\theta)$ is also invariant under $\iota$.
\end{Lemma}
This has several implications.

\subsection{Fixed points existing by symmetry}
\label{sec:symmetry}
In some cases there are distributions $\phi \in \PP(S)$ which are invariant (that is, fixed points)  ``by symmetry".
In particular, the omnipresent examples mentioned earlier:

\noindent
(i) The distribution $\delta_s$ degenerate at one point $s$;

\noindent
(ii)  The uniform two-point distribution 
$\delta_{s_1,s_2} = \frac{1}{2}(\delta_{s_1} + \delta_{s_2})$;

\smallskip \noindent
But there are further examples: 

\noindent
(iii) The Haar probability measure on a compact group $S$ with a metric invariant under the group action.

\noindent
(iv) On a finite space $S$,
a sufficient condition for the uniform distribution to be invariant is that 
$S$ is {\em transitive}, that is if for each pair $s, s^\prime$ there is an isometry taking $s$ to $s^\prime$. 
This is equivalent to the finite case of Haar measure.
But for finite $S$ a weaker condition suffices, because all that matters is the {\em rank matrix} -- see section \ref{sec:non-uniform}.

Now in those cases the distribution is invariant for all $\pi_{j,k}$.
So the question becomes: 
 for a particular $(j,k)$, are there invariant distributions with full support, other than those ``forced by symmetry" as above?

%Of course in any given case there is the equation for a fixed point -- see e.g. (\ref{1222}) and (\ref{Bin:FP}) -- but the issue is whether there exist solutions.

\subsection{Symmetry in the limit}
If $\theta  \in \PP(S)$ is invariant under an isometry $\iota$, then by Lemma \ref{L:sym} each iterate $\pi_{j,k}^n(\theta)$ is invariant 
and so we expect (recall that we do not know that $\pi_{j,k}$ is continuous) that a limit $\phi$ will be invariant.
This observation will be relevant for the formulation of conjectures.

\subsection{Fixed points for $\pi_{1,2}$ and $\pi_{2,2}$}
Perhaps the basic example of a compact space is the unit interval $[0,1]$, and for that space we will give a detailed simulation study of the iterative process in section \ref{sec:interval}.
Figure \ref{Fig:U122} 
shows the observed behavior of iterates of $\pi_{1,2}$ and $\pi_{2,2}$ from the  initial uniform distribution.
As observed above, limits of such iterates should be invariant under the reflection isometry $x \to 1-x$.
We see the former converging to $\delta_{1/2}$ and the latter to $\delta_{0,1}$.
We believe this qualitative difference will hold very generally --- see the related  Conjectures \ref{Con:12} and \ref{C:Euclid}. 
So the following general result seems very counter-intuitive.

 \begin{figure}[ht]
\includegraphics[width=2.3in]{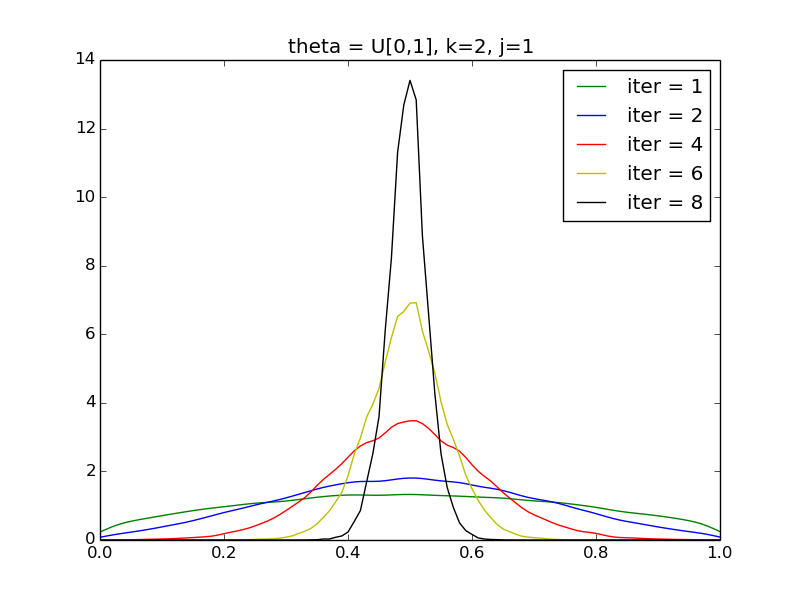}   \hspace*{-0.2in}
\includegraphics[width=2.3in]{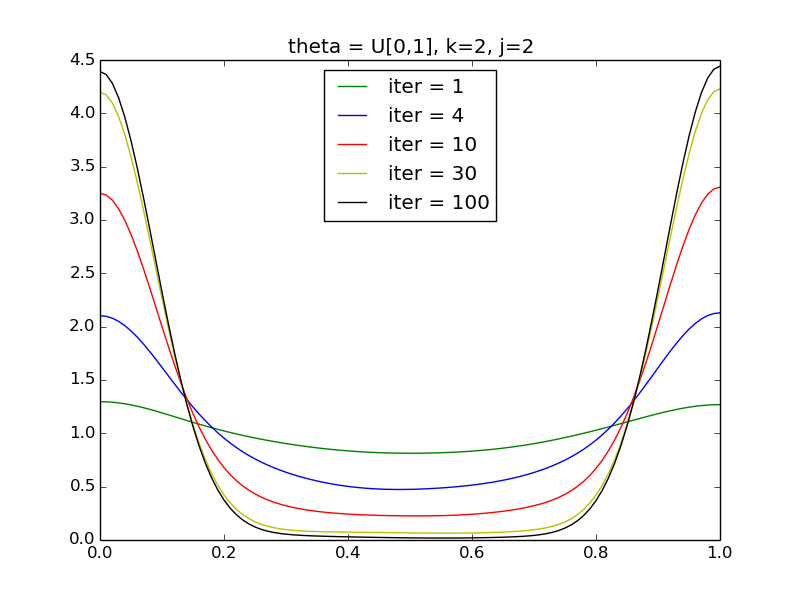} 
\caption{Iterates of $\pi_{1,2}$ (left) and $\pi_{2,2}$ (right) on the unit interval from uniform initial distribution.}
\label{Fig:U122}
\end{figure}

\begin{Theorem}[\cite{mappings_short} Theorem 3]
\label{P:k=2}
For every compact metric space $S$, the set of invariant distributions for $\pi_{1,2}$ is the same as the set of invariant distributions for $\pi_{2,2}$.
\end{Theorem}

\section{Two elements -- the binomial case} 
\label{sec:binomial}
One might suppose that the case of a 2-element set $S = \{a,b\}$ would be trivial, but it is not.
Parametrizing a distribution $\theta$ on $S$ by $p := \theta(a)$, we view the mapping $\pi_{j,k}: \PP(S) \to \PP(S)$  as a mapping $\pi_{j,k}: [0,1] \to [0,1]$ defined as follows.
In the associated 2-state Markov chain, the transition probabilities are
\[ k_{a,b} = \Pr(\Bin(k,p) < j); \quad k_{b,a} = \Pr(\Bin(k,p) > k - j) \]
for Binomial random variables.  From the stationary distribution we find 
\[
\pi_{j,k}(p) = \frac{\Pr(\Bin(k,p) > k - j) } {\Pr(\Bin(k,p) > k - j) + \Pr(\Bin(k,p) < j)   }  .
\]
So a fixed point is a solution of the equation
\begin{equation}
\pi_{j,k}(p) = p .
\label{Bin:FP}
\end{equation}
We know by symmetry that $p = 0, p = 1/2, p =1$ are fixed points; are there others?
By symmetry it is enough to consider $0 < p < 1/2$.

We have not tried to find solutions analytically, but we will show results of numerical calculations of the iterates  $ \pi^n_{j,k}(p), n = 1,2,3, \ldots$.
For a given $(k,j)$, we observe three possible types of qualitative behavior: 

\noindent
(i) $ \pi^n_{j,k}(p) \to 0 \mbox{ as } n \to \infty, \mbox{ for all }  0 < p < 1/2$.

\noindent
(ii)  $ \pi^n_{j,k}(p) \to 1/2 \mbox{ as } n \to \infty, \mbox{ for all }  0 < p < 1/2$.

\noindent
(iii)  There exists a critical value $p_{crit} \in (0,1/2)$ such that 

$p_{crit}$ is invariant 

and $ \pi^n_{j,k}(p) \to 0 \mbox{ as } n \to \infty, \mbox{ for all }  0 < p < p_{crit}$ 

and  $ \pi^n_{j,k}(p) \to 1/2 \mbox{ as } n \to \infty, \mbox{ for all }  p_{crit}  < p < 1/2$.

\noindent
For us,  (iii) is the interesting  case.
It first arises with $k = 5, j = 4$, as shown in Figure \ref{Fig:2-point,k=5,j=4}.  
We see the critical value $p_{crit} = 0.17267...$.
We also see this is an {\em unstable} fixed point.

 \begin{figure}[ht]
\includegraphics[width=2.5in]{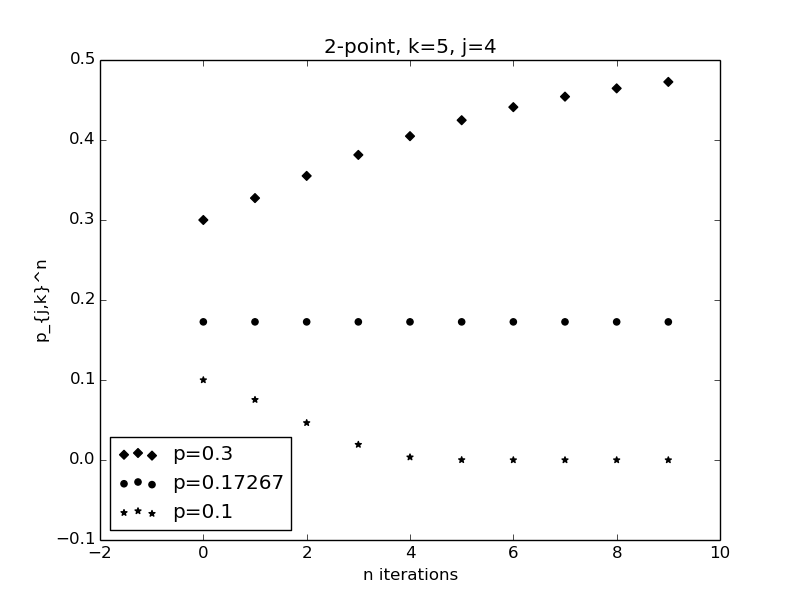}
\includegraphics[width=2.5in]{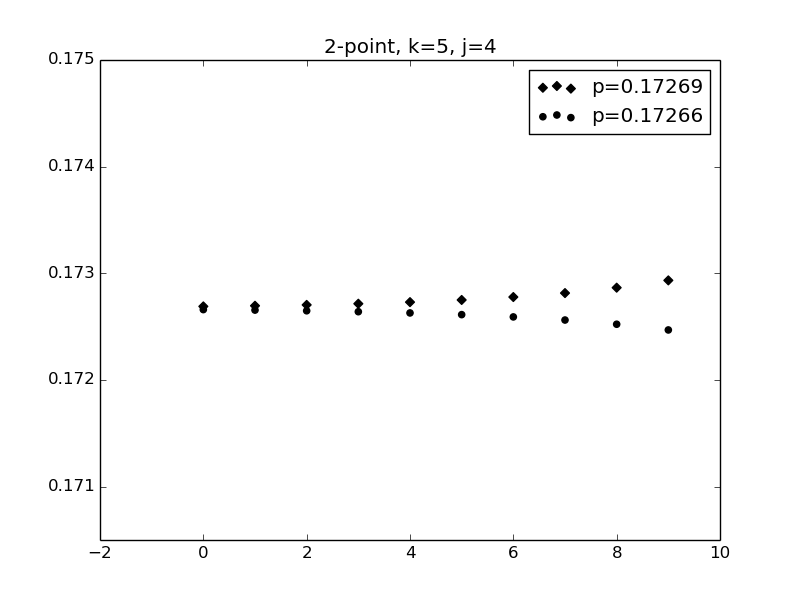}
\caption{$S = \{a,b\}; \ k = 5, j = 4$.  Iterates $n = 0,1,2,\ldots,10$. Left panel shows type (iii) behavior, Right panel shows the unstable fixed point at 0.17267.}
\label{Fig:2-point,k=5,j=4}
\end{figure}

\begin{table}
\caption{$S = \{a,b\}$ and $2 \le k \le 9$.   The values of $j$ with each type of behavior, and (critical values) in type (iii) behavior.}
\begin{center}
$\begin{array}{cccc}
k & (i) & (iii) & (ii) \\
\hline 
2 & 1 && 2 \\
3 & [1,2] && 3\\
4 & [1 - 3] & & 4 \\
5 & [1 - 3] & 4\  (0.17267) & 5\\
6 & [1 - 4] & 5\  (0.09558) & 6 \\
7 & [1 - 5] & 6\  (0.06276) & 7\\
8 & [1 - 5] & 6\  (0.26405) & [7,8] \\
9 & [1 - 6] & 7\  (0.18884); \ \  8\  (0.03364) &  9 
\end{array}$
\end{center}
\label{Table:1}
\end{table}

Table \ref{Table:1} shows the type of behavior -- types (i) or (ii) or (iii) above -- for all pairs $(j,k)$ with $k \le 9$.
One take-away message is that for $S = \{a,b\}$ there exist some $(j,k)$ for which $\pi_{j,k}$ has fixed points in addition to those existing by symmetry, but these fixed points are unstable.
Another take-away message is that, except for these additional fixed points, the limit of the iterative process is either $\delta_{a}$ or
$\delta_{b}$ or $\delta_{a,b}$,  the latter only when $j/k$ is somewhat close to $1$.
Of course the 2-point space may be very special.
What properties extend to other $S$?

  \section{Finite spaces}
  \label{sec:finite}
 \subsection{Non-uniform invariant distributions}
 \label{sec:non-uniform}
 In our initial investigations, via simulations of the iterative process on different spaces, we found that
 the iterates always converged to  one- or two-point support invariant distributions, 
 $\delta_s$ or $\delta_{s_1,s_2}$.  
 As mentioned before, on a space with sufficient symmetry the uniform distribution will be invariant. 
 So perhaps the existence of non-uniform invariant distributions  on the 2-element space for $k \ge 5$ is an anomaly?
 These observations were consistent with the possibility that {\bf all} invariant distributions supported on more than two elements are uniform on their support.
 But a numerical search revealed the following counter-example, in the simplest case $j = 1, k = 2$.

 A finite metric space can be represented by the matrix $D$ of distances $d(i,j)$.
 By taking all the non-zero distances to be between $1$ and $2$, the triangle inequality is automatically satisfied.
 Consider the example of a 5-element space with distance matrix
 \[ D = 
 \begin{pmatrix}
 0 &   1.714 &  1.341 &  1.656 & 1.74 \\
 1.714 & 0 &   1.298  & 1.794 & 1.03 \\
 1.341 & 1.298  &0&    1.715  & 1.844\\
 1.656 & 1.794 & 1.715 &0  &   1.524\\
 1.74 & 1.03 & 1.844 & 1.524 & 0   
 \end{pmatrix}
 \]
What matters for our purposes, assuming as in this example that all distances are distinct\footnote{Precisely, all distances $d(i,j), 1 \le i < j \le n$ are distinct.},  is the {\em rank matrix} $R$, where $r(i,j) = 4$ means that $d(i,j)$ is the 4'th smallest of $\{d(i,1), d(i,2), \ldots, d(i,|S|)\}$.
 For the distance matrix $D$ above,  the rank matrix is
  \[ R = 
 \begin{pmatrix}
 1 & 4 & 2 & 3 & 5 \\
 4 & 1 & 3 & 5 & 2 \\
 3 & 2 & 1 & 4 & 5 \\
 3 & 5 & 4 & 1 & 2 \\
 4 & 2 & 5 & 3 & 1
 \end{pmatrix}
 \]
By numerical calculation, for $\pi_{1,2}$ on this space there is an invariant distribution 
\[
\theta \approx (0.149 \ 0.188\  0.203\  0.298 \ 0.162) 
\]
for which the transition matrix is
  \[ K \approx 
 \begin{pmatrix}
 0.276 & 0.097 & 0.304 & 0.297 & 0.026\\
 0.111 & 0.341 & 0.222 & 0.089 & 0.237 \\
 0.159 & 0.265 & 0.365 & 0.185 & 0.026 \\
 0.139 & 0.036 & 0.118 & 0.507 & 0.201 \\
 0.083 & 0.28  & 0.041 & 0.298  & 0.298
 \end{pmatrix}
 \]
This example was found by simulating random distance matrices $D$, obtaining the rank matrix $R$, and then numerically solving for invariant distributions $\theta$ until finding a solution with full support.
Note this involved non-linear equations: we need to solve $\theta K = \theta$ but here $K$ depends on $\theta$, for instance for $\pi_{1,2}$ 
\[ k(i,i) = 1 - (1- \theta(i))^2 \]
\[ \mbox{ if } r(i,j) = 5 \mbox{ then } k(i,j) = \theta^2(j) . \]
Note also that for $|S| = 5$ there are only a finite number of possible rank matrices $R$, so this counter-example is not like a  counter-example depending on a real parameter taking a specific value.

To be precise, let us define a {\em rank matrix} $R$ to be a matrix where each row consists of $\{1, 2, \ldots,|S|\}$ in some order and each diagonal element $r_{ii} = 1$.
We are concerned with rank matrices which are {\em feasible} in that they arise from some distance matrix $D$ with distinct entries.
We do not know precise conditions to determine whether a given rank matrix is feasible.
Also, we do not have any general conjecture about conditions under which  non-uniform invariant distributions exist.

 We observed that the invariant distribution in Figure \ref{Fig:2-point,k=5,j=4} was unstable. 
 Here is another example. 
 Consider again $\pi_{1,2}$ and the rank matrix (easily checked to be feasible) 
  \begin{equation}
    R = 
 \begin{pmatrix}
 1 & 2 & 4 & 3  \\
 2 & 1 & 3 & 4  \\
 4 & 2 & 1 & 3 \\
 2 & 4 & 3 & 1 
 \end{pmatrix}
 \label{R2}
 \end{equation}
 for which the invariant distribution is
 \[
 \theta = (\sfrac{1}{6} \ \sfrac{1}{6}  \ \sfrac{2}{6}  \ \sfrac{2}{6} )
 \]
and the transition matrix is
  \[ K =
 \begin{pmatrix}
 11/36 & 1/4 & 1/9 & 1/3  \\
 1/4 & 11/36 & 1/3 & 1/9 \\
 1/36 & 7/36 & 5/9 & 2/9 \\
 7/36 & 1/36 & 2/9 & 5/9
 \end{pmatrix}
 \] 
 This is unstable under a generic small perturbation of the invariant distribution, as illustrated in Figure  \ref{Fig:unstable_5}.

%Figure  \ref{Fig:unstable_5} illustrates instability in a different example with $|S| = 4$ and an invariant distribution $\theta_1 = \theta_2 = 1/3, \ \theta_3 = \theta_4 = 2/6$ arising from a certain partial symmetry.

 \begin{figure}[ht]
 \includegraphics[width=4.7in]{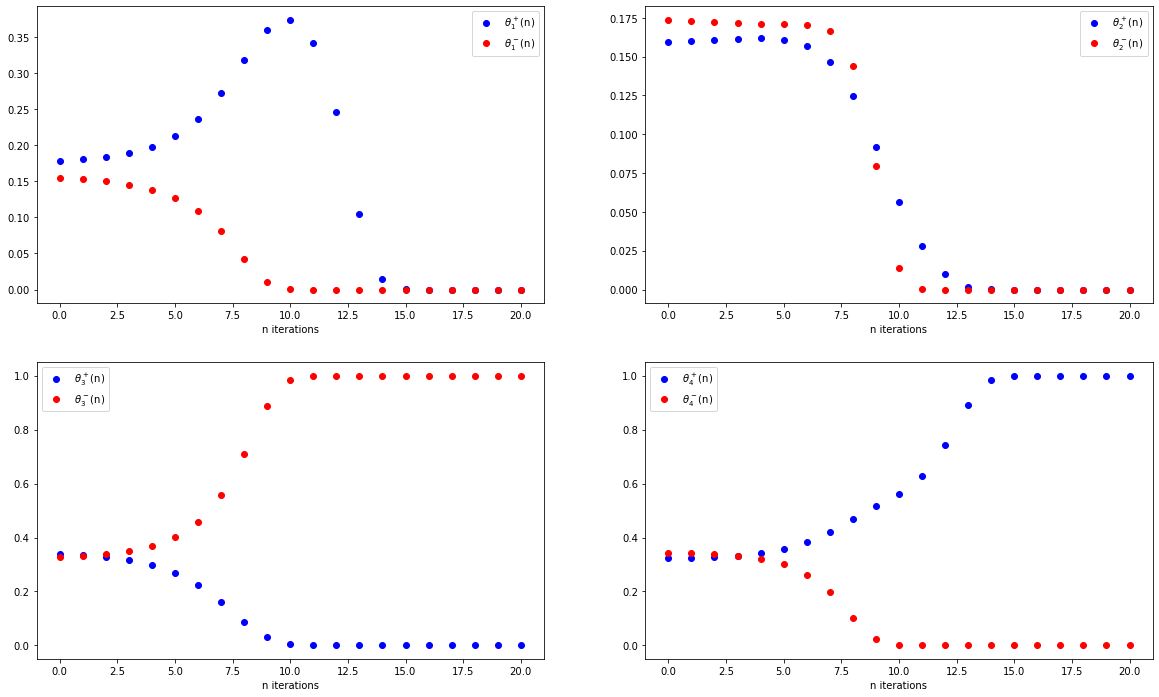}
\caption{$|S| = 4$, rank matrix $R$ at (\ref{R2}), $\pi_{1,2}$.      Unstable behavior of iterates $\theta_i(n), n = 0,1,2,\ldots, 20$ in panels $i = 1,2,3,4$, starting 
from two different initial distributions  $\theta^+, \theta^-$ near the fixed point.}
\label{Fig:unstable_5}
\end{figure}

However if instead of $\pi_{1,2}$ we consider $\pi_{2,2}$ for the same rank matrix (\ref{R2}), then
by Theoremj \ref{P:k=2} we have same invariant
distribution $\theta =  (\sfrac{1}{6} \ \sfrac{1}{6} \ \sfrac{2}{6} \ \sfrac{2}{6})$,  but 
if we take an initial distribution of the form 
$(\sfrac{1}{6} + \eps \ \ \sfrac{1}{6} + \eps \ \ \sfrac{2}{6} - \eps \ \ \sfrac{2}{6} - \eps)$ 
then iterates do converge to the fixed point.
This example seems to depend on a certain ``partial symmetry" property of the rank matrix which is copied to the perturbed initial distribution.
We do not know a precise formulation of the ``partial symmetry" idea here.
We avoid that issue as follows,  formulating  the strongest conjecture to which we do not have a counter-example.
 \begin{Conjecture}
 \label{Con:12}
 On any finite space $S$ with distinct distances, for {\em almost all}  initial distributions $\theta$ 
 the iterative process 
 $(\pi^n_{j,k}(\theta), n \ge 1)$ 
 converges to a limit of the form  $\delta_s$ or $\delta_{s_1,s_2}$.  
 \end{Conjecture}
 Here {\em almost all} is with respect to the natural uniform measure on the simplex of probability distributions on $S$.
 This would imply that our original plan for finding invariant distributions, by iterating from some haphazardly-chosen initial distribution, would not be effective in finding {\bf all} invariant distributions.

 \subsection{Uniform invariant distributions}
 We noted earlier that  strong symmetry conditions on a compact metric space $S$ would be sufficient to imply that the uniform distribution was invariant.
 For finite $S$, it is clearly sufficient that the rank matrix $R$ is a Latin square.
 To fit our context we need the rank matrix to be feasible, so we mention the following result:
 we do not know if it is new.
 \begin{Lemma}
 \label{L:rank}
 If a rank matrix $R$ is a Latin square, then $R$ is feasible if and only if $R$ is a symmetric matrix.
 \end{Lemma}

 \subsection{Finite sets in Euclidean space}
 \label{sec:2-dim}
 As another exploratory example of finite $S$, Figure \ref{Fig:9points} shows results of simulations %\footnote{Simulations by Madelyn Cruz.} 
 for a 9-point set. The points are located as a perturbed $3 \times 3$ pattern, marked as $\bullet$ in the Figure, with Euclidean distance in the plane.
  We fix $k = 10$ and consider each $j = 1,2,\ldots,10$.
 Results for three different initial distributions are shown, indicated by the three different colors.
 The colored vertical lines in each square indicate the initial probabilities at the associated point $\bullet$.
 The results show behavior analogous to the previous settings, in that the ``fixed point" limit of the iterative procedure 
 is either a unit mass $\delta_s$ or $\delta_{s_1,s_2} = \frac{1}{2}(\delta_{s_1} + \delta_{s_2})$.
 In Figure \ref{Fig:9points}, a $\delta_s$ limit for the $j$'th nearest process is represented by a colored $j$ in square $s$, or a limit $\delta_{s_1,s_2}$ is
 represented by a colored $(j)$ in both square $s_1$ and square $s_2$.  
 For instance, for $j = 8$ the limit is (for each initial distribution color)  uniform on some two points: top left and top right for blue,   bottom left and middle right for red, and top center and bottom left for brown.
 
\bigskip

  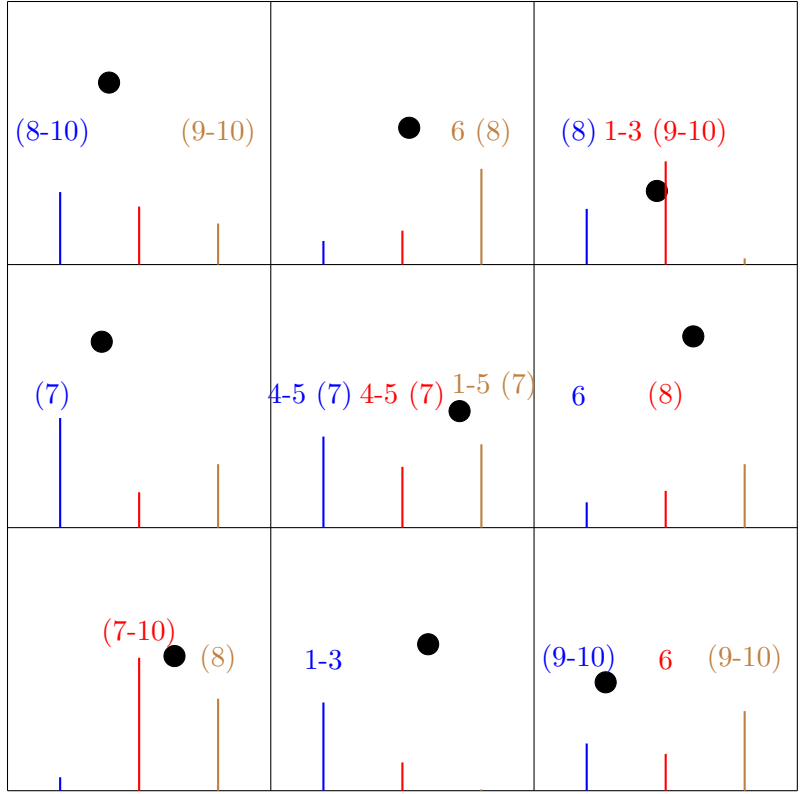
\begin{figure}[ht]
 \begin{tikzpicture}[x=3.5cm,y=3.5cm]
 \draw[thin,black] (1.75,0.75) -- (1.75,3.75) ;
 \draw[thin,black] (2.75,0.75) -- (2.75,3.75) ;
  \draw[thin,black] (0.75,0.75) -- (0.75,3.75) ;
 \draw[thin,black] (3.75,0.75) -- (3.75,3.75) ;
 \draw[thin,black] (0.75,1.75) -- (3.75,1.75) ; 
 \draw[thin,black] (0.75,2.75) -- (3.75,2.75) ; 
  \draw[thin,black] (0.75,0.75) -- (3.75,0.75) ; 
 \draw[thin,black] (0.75,3.75) -- (3.75,3.75) ; 
 \filldraw (1.3843, 1.2619) circle [radius=0.04];
 \filldraw (1.108, 2.4567) circle [radius=0.04];
  \filldraw (1.1358, 3.4418) circle [radius=0.04];
  \filldraw (2.2758, 3.27) circle [radius=0.04];
  \filldraw (3.022, 1.162) circle [radius=0.04];
 \filldraw (3.3549, 2.477) circle [radius=0.04];
  \filldraw (2.4671, 2.1929) circle [radius=0.04];
  \filldraw (2.3478, 1.3067) circle [radius=0.04];
 \filldraw (3.2165, 3.0299) circle [radius=0.04];
 
 \draw[thick,blue] (0.95,0.75) -- (0.95,0.75 + 2*0.0256);
  \draw[thick,blue] (0.95,1.75) -- (0.95,1.75 + 2*0.2083);
 \draw[thick,blue] (0.95,2.75) -- (0.95,2.75 + 2*0.1378);
  \draw[thick,blue] (1.95,0.75) -- (1.95,0.75 + 2*0.1677);
  \draw[thick,blue] (1.95,1.75) -- (1.95,1.75 + 2*0.1731);
 \draw[thick,blue] (1.95,2.75) -- (1.95,2.75 + 2*0.0449);
  \draw[thick,blue] (2.95,0.75) -- (2.95,0.75 + 2*0.0897);
  \draw[thick,blue] (2.95,1.75) -- (2.95,1.75 + 2*0.0481);
 \draw[thick,blue] (2.95,2.75) -- (2.95,2.75 + 2*0.1058);
 
  \draw[thick,red] (1.25,0.75) -- (1.25,0.75 + 2*0.2527);
  \draw[thick,red] (1.25,1.75) -- (1.25,1.75 + 2*0.0672);
 \draw[thick,red] (1.25,2.75) -- (1.25,2.75 + 2*0.1102);
  \draw[thick,red] (2.25,0.75) -- (2.25,0.75 + 2*0.0538);
  \draw[thick,red] (2.25,1.75) -- (2.25,1.75 + 2*0.1156);
 \draw[thick,red] (2.25,2.75) -- (2.25,2.75 + 2*0.0645);
  \draw[thick,red] (3.25,0.75) -- (3.25,0.75 + 2*0.0699);
  \draw[thick,red] (3.25,1.75) -- (3.25,1.75 + 2*0.0699);
 \draw[thick,red] (3.25,2.75) -- (3.25,2.75 + 2*0.1962);
 
   \draw[thick,brown] (1.55,0.75) -- (1.55,0.75 + 2*0.1749);
  \draw[thick,brown] (1.55,1.75) -- (1.55,1.75 + 2*0.1206);
 \draw[thick,brown] (1.55,2.75) -- (1.55,2.75 + 2*0.0780);
  \draw[thick,brown] (2.55,0.75) -- (2.55,0.75 + 2*0.0024);
  \draw[thick,brown] (2.55,1.75) -- (2.55,1.75 + 2*0.1584);
 \draw[thick,brown] (2.55,2.75) -- (2.55,2.75 + 2*0.182);
  \draw[thick,brown] (3.55,0.75) -- (3.55,0.75 + 2*0.1513);
  \draw[thick,brown] (3.55,1.75) -- (3.55,1.75 + 2*0.1206);
 \draw[thick,brown] (3.55,2.75) -- (3.55,2.75 + 2*0.0118);

 \node[blue] at (0.92,2.25) {(7)};
 \node[blue] at (0.92,3.25) {(8-10)};
 \node[blue] at (1.9,2.25) {4-5 (7) };
 \node[blue] at (1.95,1.25) {1-3 }; 
 \node[blue] at (2.92,3.25) {(8)};
 \node[blue] at (2.92,2.25) {6};
  \node[blue] at (2.92,1.25) {(9-10)};
  
 \node[red] at (3.25,3.25) {1-3 (9-10)};
 \node[red] at (2.25,2.25) {4-5 (7) };
 \node[red] at (3.25,1.25) {6};
 \node[red] at (1.25,1.35) {(7-10)};
 \node[red] at (3.25,2.25) {(8)};
 
 \node[brown] at (2.6,2.29) {1-5 (7) };
 \node[brown] at (2.55,3.25) {6 (8)};
 \node[brown] at (1.55,1.25) {(8)};
 \node[brown] at (1.55,3.25) {(9-10)};
 \node[brown] at (3.55,1.25) {(9-10)};
 \end{tikzpicture}
 \caption{A 9-point set in the plane.  See text for explanation.}
\label{Fig:9points}
\end{figure}
 
 \bigskip
 
 What we observe in these results (in each of 11 trials from different initial distributions, extending the 3 trials shown in the figure) is that
 \begin{quote} 
 (*) for $j \le 6$ the limit is always some $\delta_s$ whereas for $j \ge 7$ the limit is always
 some $\delta_{s_1,s_2}$. But the precise limit -- which $s$ or $s_1, s_2$ -- depends both on $j$ and the initial distribution.
 \end{quote}
 
 Remarkably, simulations for an analogous perturbed $5 \times 5$ pattern $S$ show exactly the same behavior described in (*),
 and so do  simulations for an analogous perturbed $3 \times 3 \times 3$ pattern in three dimensions.

 Here is one conjecture.
 \begin{Conjecture}
 \label{C:Euclid}
 For finite $S$ embedded in the plane, with Euclidean distance, consider $\pi_{1,2}$.  
 Then for almost all initial distributions,  the limit of the iterative process is  some $\delta_s$.
 \end{Conjecture} 
 Indeed this might be true on every compact metric space $S$, for some appropriate notion of ``almost all".

 \section{A discrete tree space}
 \label{sec:tree}

 \setlength{\unitlength}{0.1in}
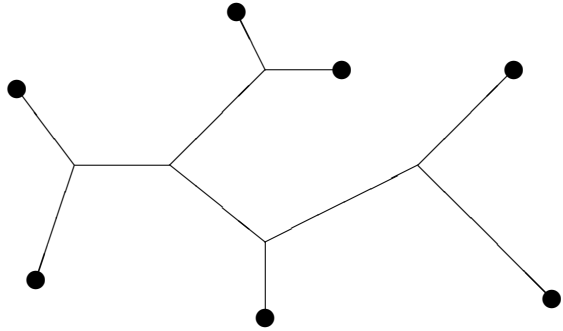
\begin{figure}
\begin{picture}(30,22)(-8,0)
\put(10,10){\line(-1,0){5}}
\put(10,10){\line(1,1){5}}
\put(10,10){\line(5,-4){5}}
\put(15,6){\line(2,1){8}}

\put(5,10){\line(-3,4){3}}
\put(2,14){\circle*{1}}
\put(5,10){\line(-1,-3){2}}
\put(3,4){\circle*{1}}

\put(15,15){\line(-1,2){1.5}}
\put(13.5,18){\circle*{1}}
\put(15,15){\line(1,0){4}}
\put(19,15){\circle*{1}}

\put(15,6){\line(0,-1){4}}
\put(15,2){\circle*{1}}

\put(23,10){\line(1,1){5}}
\put(28,15){\circle*{1}}
\put(23,10){\line(1,-1){7}}
\put(30,3){\circle*{1}}
\end{picture}
\caption{A BTL space $S$ with $|S| = 7$.}
\label{Fig:BTL}
\end{figure}

In \cite{mappings_short}  we study  binary tree leaves (BTL), illustrated in Figure \ref{Fig:BTL}, as a class of finite spaces.
Here $S$ is the finite set of leaves; 
the edges have lengths which serve  to determine the distance between two leaves  as the length of the unique path 
between them; the edges also define $|S| - 2$ branchpoints.
To ``break symmetry" we assume
\begin{equation}
\mbox{ all distances $(d(s_i,s_j), j \ne i)$ are distinct.}
\label{distinct}
\end{equation}
We claim that, 
as suggested by the general picture from numerics, 
 for $k = 2$ there are no invariant measures other
than the omnipresent ones.  
%In the case of $|S| = 3$ this is equivalent to Theorem \ref{T:3}.
An invariant measure supported on a subset of leaves is an invariant measure on the induced spanning tree of that subset,
so to prove that claim it suffices to prove
\begin{Theorem}[\cite{mappings_short} Theorem 7]:
\label{Conj}
On a BTL space $S$ with $|S| \ge 3$ and satisfying \eqref{distinct}, and for $k=2$,
there are no invariant measures with full support.
\end{Theorem}
 
 As with the analogous result below on the unit interval, the proof is a
 ``proof by contradiction" depending on the specific space $S$, 
 which does not suggest possible general proofs.

 \section{The unit interval}
 \label{sec:interval}
 Here we take $S$ to be the unit interval $[0,1]$.
 We will show simulations and discuss what they suggest regarding convergence to fixed points.
These are Monte Carlo simulations:  a distribution is represented as a population of (typically 500,000) sample points.
 In the Figures, densities are drawn via Gaussian KDE, which becomes inaccurate for sharply peaked distributions,
 and also tends to be inaccurate at endpoints.

 \subsection{Initial distribution $U[0,1]$.}
 \label{sec:U01}
 We first study, by simulation, iterates starting from the uniform distribution $U[0,1]$. 
 Because $U[0,1]$ is symmetric about $1/2$, all iterates must be symmetric about $1/2$; in section \ref{sec:non-u} we will study non-symmetric initial distributions,
to investigate whether symmetry might be forcing some non-generic behavior.
 
 \begin{figure}[ht]
\includegraphics[width=1.7in]{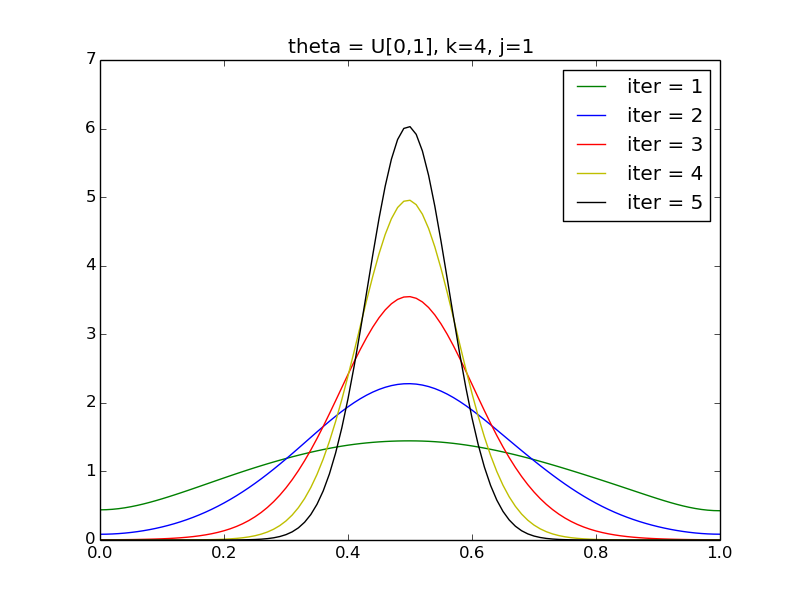}  \hspace*{-0.2in}
\includegraphics[width=1.7in]{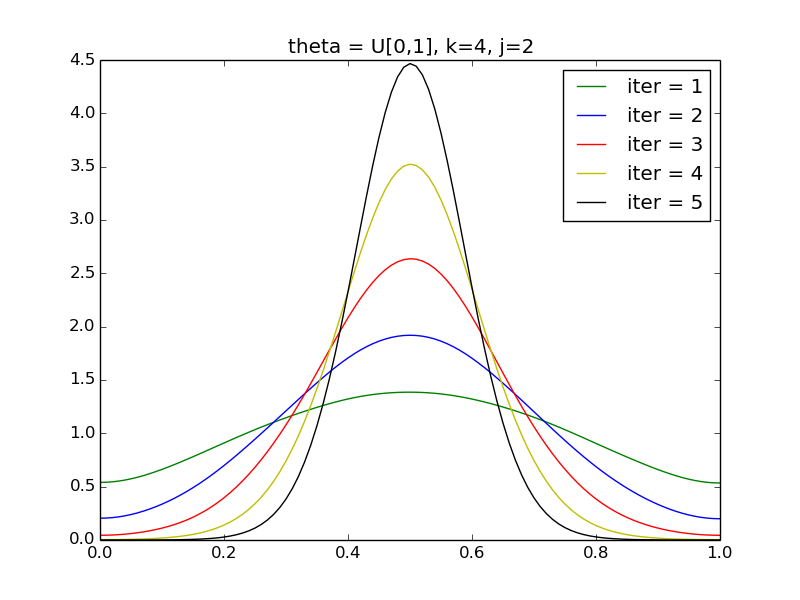}   \hspace*{-0.2in}
\includegraphics[width=1.7in]{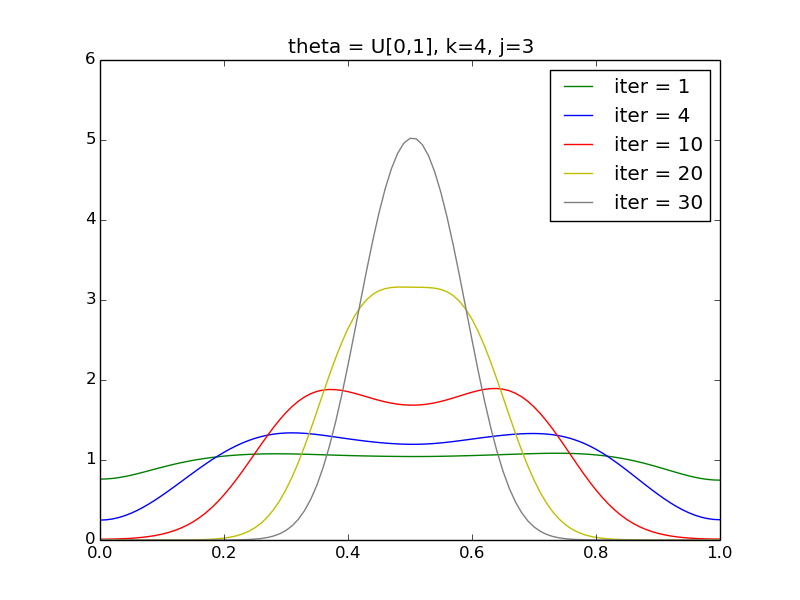} 
\caption{Iterates  from $U[0,1]$; \   $k = 4, j = 1,2,3.$}
\label{Fig:U1}
\end{figure}

 \begin{figure}[ht]
 \begin{center}
\includegraphics[width=1.7in]{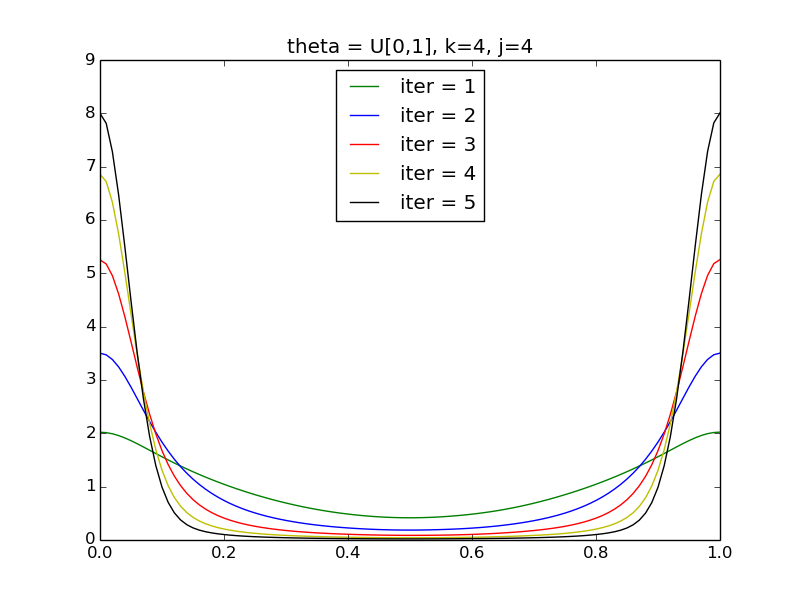}
\end{center}
\caption{Iterates from $U[0,1]$; \  $k = 4, j = 4$}
\label{Fig:U2}
\end{figure}

\noindent
First let us show the case $k = 4$. 
Figure \ref{Fig:U1} shows the first few iterates of $U[0,1]$ for $j = 1, 2, 3$.
Note that, here and throughout, the vertical scale and the numbers of iterations shown may not be the same from one panel to the next.
What we see strongly suggests that the iterates are converging, quickly for $j = 1$ but rather slowly for $j = 3$, 
toward the degenerate distribution $\delta_{1/2}$.  
This is strongly supported by examining the standard deviations of the iterates, shown on log scale in Figure \ref{Fig:sd1}, which will be discussed below. 
In contrast, Figure \ref{Fig:U2} for $j = 4$ strongly suggests that the iterates are converging quickly 
toward the mixture  $\delta_{0,1}$.  
These two ``extreme" behaviors -- convergence to $\delta_{1/2}$ for smaller $j$ or to $\delta_{0,1}$ for larger $j$ -- appear to hold for all $k$. 
Figure \ref{Fig:U3} shows simulations for $k = 6$, and Table \ref{Table:2} shows which behavior appears to hold in simulations for each pair $(j,k)$ with $k \le 9$.

 \begin{figure}[ht]
\includegraphics[width=1.7in]{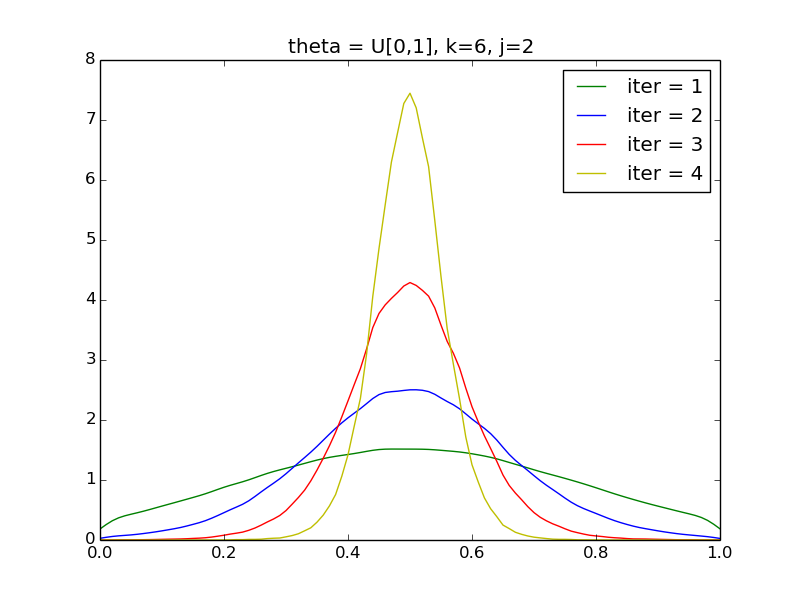}  \hspace*{-0.2in}
\includegraphics[width=1.7in]{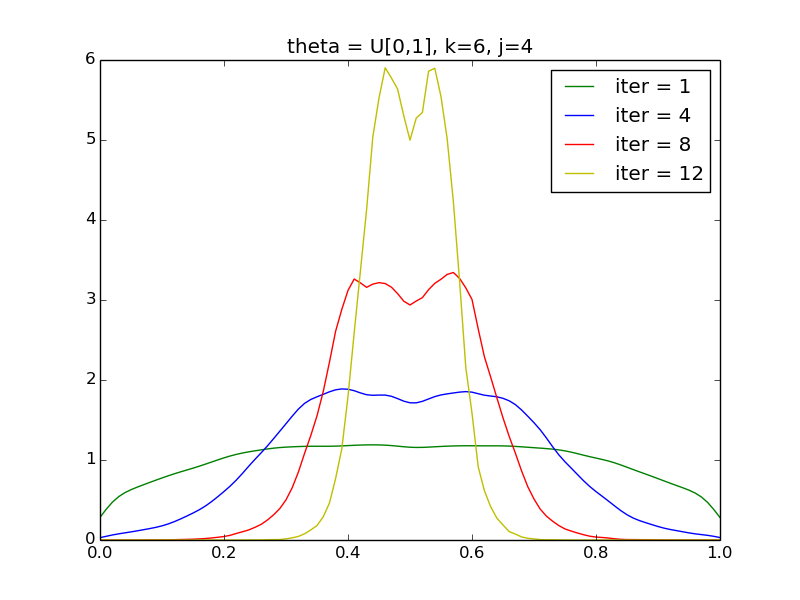}   \hspace*{-0.2in}
\includegraphics[width=1.7in]{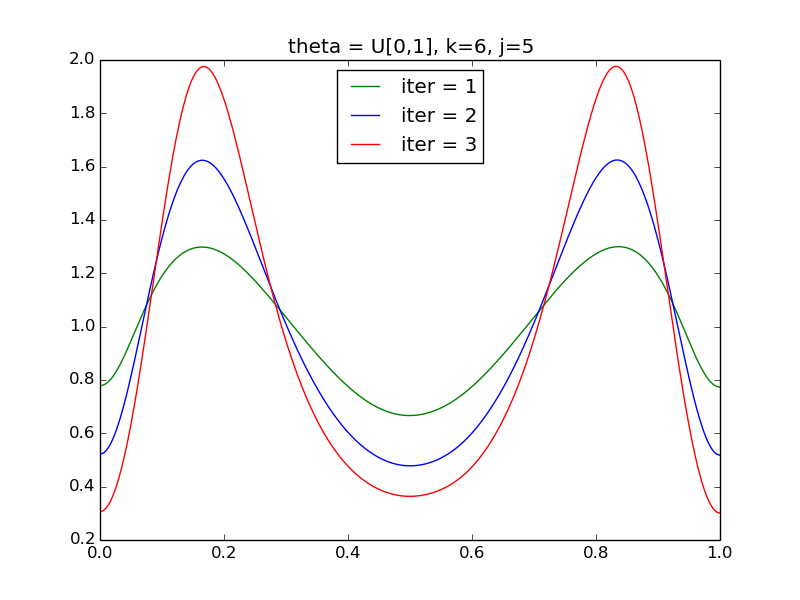} 
\caption{Iterates  from $U[0,1]$; \   $k = 6, j = 2,4,5.$}
\label{Fig:U3}
\end{figure}

\begin{table}[ht]
\caption{Conjectured limits of iterates from $U[0,1]$; \  the values of $j$ with each type of behavior.}
\begin{center}
$\begin{array}{ccc}
k & \to \delta_{1/2} & \to \delta_{0,1} \\
\hline 
2 & 1 & 2 \\
3 & [1,2] & 3\\
4 & [1 - 3] & 4 \\
5 & [1 - 4] &  5\\
6 & [1 - 4] &  [5, 6] \\
7 & [1 - 5] &   [6,7] \\
8 & [1 - 6] &   [7,8] \\
9 & [1 - 7] &    [8,9] 
\end{array}$
\end{center}
\label{Table:2}
\end{table}

The pattern in Table \ref{Table:2} is reminiscent of the pattern seen earlier for $S = \{a,b\}$, 
but one must be careful about the comparison. 
In Figure \ref{Fig:2-point,k=5,j=4} for $S = \{a,b\}$ each $\bullet$ represents a probability distribution, whereas in Figures \ref{Fig:U1} - \ref{Fig:U3}
for $S = [0,1]$ a probability distribution is represented by the estimated density curve.  
Figure \ref{Fig:2-point,k=5,j=4} shows iterates of the same $\pi_{j,k}$ from different initial distributions, all on the same graphic. 
 But for $S = [0,1]$ we will show different
graphics for the different initial distributions -- that is, we should compare Figures  \ref{Fig:U1} - \ref{Fig:U2} with Figures \ref{Fig:T1} -  \ref{Fig:T2} below.

Are there any cases on $[0,1]$ where the iterations of $\pi_{j,k}$ from the uniform initial distribution converge to a limit other than 
$\delta_{1/2}$ or $\delta_{0,1}$?
Consider the case $k=9, j=7$ in Figure \ref{Fig:U4} 
(the case $k = 6, j = 5$ in Figure \ref{Fig:U3} is similar).
Once two tight peaks are formed, one can see heuristically how the Markov chain behaves.
Suppose $s$ is near the left peak.
Amongst the $9$ samples, $4$ or $5$ will typically be near the right peak, in which case the 7'th closest to $s$ will be the 2nd of $4$ 
or the 3rd of $5$,
So on average the chain jumps to slightly to the left of center of the right peak.
This is consistent with what we observe in Figure \ref{Fig:U4}: the peaks grow in height and move slowly in the direction of $1/2$.
It may be possible, in a case like this, that the iterates converge to a limit of the form $\delta_{s,1-s}$ for $0 < s < 1/2$ rather than to $\delta_{1/2}$.
We cannot decide convincingly from our simulations.  
Thus heuristic may be more widely applicable -- see section \ref{sec:34}.

 \begin{figure}[ht]
 \begin{center}
\includegraphics[width=2.5in]{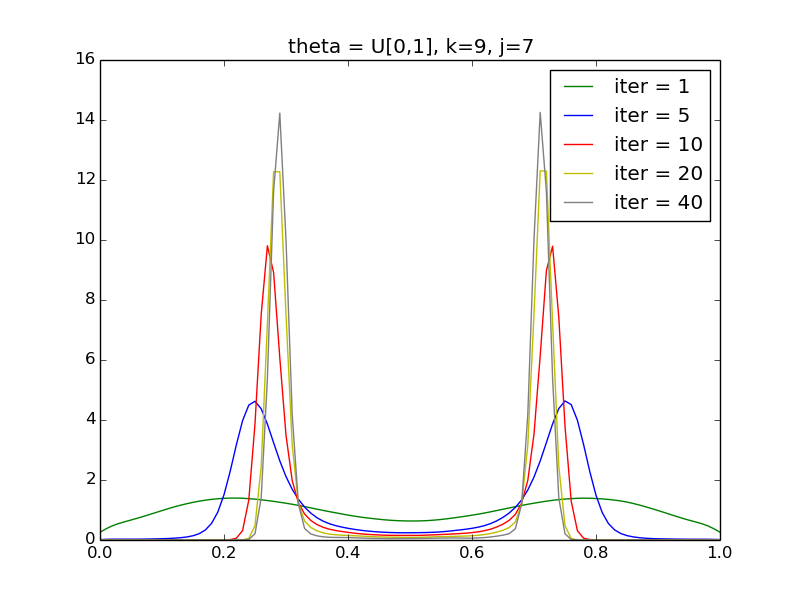}
\end{center}
\caption{Iterates from $U[0,1]$; \   $k = 9, j = 7$.}
\label{Fig:U4}
\end{figure}

Simulations also suggest a scaling limit.
Figure \ref{Fig:sd1} shows, for some $(j,k)$ for which the iterates converge to $\delta_{1/2}$  (others are similar), a remarkably precise 
geometric decrease in the s.d. as a function of the number of iterations, after the first few iterations.
This strongly suggests a certain form of asymptotic self-similarity under scaling: 
that there exists a mean-zero distribution, say $\dist(\xi_{j,k})$,  on $\Reals$ and a constant $0 < c_{j,k} < 1$ 
such that $\pi_{j,k}[\dist(\xi_{j,k})] = \dist (c_{j,k} \xi_{j,k})$. 
And that $\dist(\xi_{j,k})$ is the scaling limit of the iterates, and $c_{j,k}$ the geometric rate constant. 
Figure \ref{Fig:U5} shows renormalized (by mean and s.d.) iterates approximating $\dist(\xi_{2,4})$.

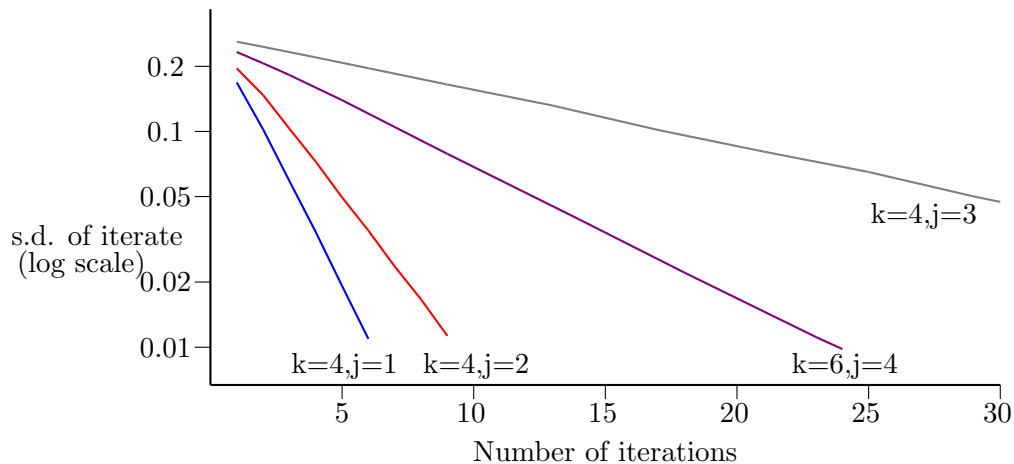
\begin{figure}[ht]
\begin{tikzpicture}[x=0.35cm,y=1.25cm]
\draw[thick,black] (0,0) -- (30,0 ) ;
\draw[thick,black] (0,0) -- (0,4) ;
\draw[thin,black] (0,{5 + ln(0.01)}) -- (-0.6,{5 + ln(0.01)}) ;
 \node at (-1.7,{5 + ln(0.01)}) {0.01} ; 
 \draw[thin,black] (0,{5 + ln(0.02)}) -- (-0.6,{5 + ln(0.02)}) ;
 \node at (-1.7,{5 + ln(0.02)}) {0.02} ; 
 \draw[thin,black] (0,{5 + ln(0.05)}) -- (-0.6,{5 + ln(0.05)}) ;
 \node at (-1.7,{5 + ln(0.05)}) {0.05} ;  
 \draw[thin,black] (0,{5 + ln(0.1)}) -- (-0.6,{5 + ln(0.1)}) ;
 \node at (-1.7,{5 + ln(0.1)}) {0.1} ; 
 \draw[thin,black] (0,{5 + ln(0.2)}) -- (-0.6,{5 + ln(0.2)}) ;
 \node at (-1.7,{5 + ln(0.2)}) {0.2} ; 
  \draw[thin,black] (5,0) -- (5,-0.15) ;
  \node at (5,-0.3) {5} ;
  \draw[thin,black] (10,0) -- (10,-0.15) ;
   \node at (9.9,-0.3) {10} ;
  \draw[thin,black] (15,0) -- (15,-0.15) ; 
 \node at (14.9,-0.3) {15} ;
 \draw[thin,black] (20,0) -- (20,-0.15) ;  
  \node at (19.9,-0.3) {20} ;
 \draw[thin,black] (25,0) -- (25,-0.15) ; 
 \node at (24.9,-0.3) {25} ;
  \draw[thin,black] (30,0) -- (30,-0.15) ;  
   \node at (29.9,-0.3) {30} ;

\draw[thick,blue] (1, {5  + ln(0.1677)}) -- (2,{5 + ln(0.1023)}) ;
\draw[thick,blue] (2, {5  + ln(0.1023)}) -- (3,{5 + ln(0.05906)}) ;
\draw[thick,blue] (3, {5  + ln(0.05906)}) -- (4,{5 + ln(0.03368)}) ;
\draw[thick,blue] (4, {5  + ln(0.03368)}) -- (5,{5 + ln(0.01915)}) ;
\draw[thick,blue] (5, {5  + ln(0.01915)}) -- (6,{5 + ln(0.01096)}) ;\
\node at (5.1,0.2) {k=4,j=1} ;

\draw[thick,red] (1, {5  + ln(0.1956)}) -- (2,{5 + ln(0.1456)}) ;
\draw[thick,red] (2, {5  + ln(0.1456)}) -- (3,{5 + ln(0.1030)}) ;
\draw[thick,red] (3, {5  + ln(0.1030)}) -- (4,{5 + ln(0.07157)}) ;
\draw[thick,red] (4, {5  + ln(0.07157)}) -- (5,{5 + ln(0.04956)}) ;
\draw[thick,red] (5, {5  + ln(0.04956)}) -- (6,{5 + ln(0.03424)}) ;
\draw[thick,red] (6, {5  + ln(0.03424)}) -- (7,{5 + ln(0.02374)}) ;
\draw[thick,red] (7, {5  + ln(0.02374)}) -- (8,{5 + ln(0.01640)}) ;
\draw[thick,red] (8, {5  + ln(0.01640)}) -- (9,{5 + ln(0.01133)}) ;
\node at (10.1,0.2) {k=4,j=2} ;

\draw[thick,gray] (1, {5  + ln(0.2602)}) -- (2,{5 + ln(0.2464)}) ;
\draw[thick,gray] (2, {5  + ln(0.2464)}) -- (4,{5 + ln(0.2202)}) ;
\draw[thick,gray] (4, {5  + ln(0.2202)}) -- (6,{5 + ln(0.1959)}) ;
\draw[thick,gray] (6, {5  + ln(0.1959)}) -- (9,{5 + ln(0.1639)}) ;
\draw[thick,gray] (9, {5  + ln(0.1639)}) -- (13,{5 + ln(0.1288)}) ;
\draw[thick,gray] (13, {5  + ln(0.1288)}) -- (17,{5 + ln(0.1016)}) ;
\draw[thick,gray] (17, {5  + ln(0.1016)}) -- (21,{5 + ln(0.08026)}) ;
\draw[thick,gray] (21, {5  + ln(0.08026)}) -- (25,{5 + ln(0.063428)}) ;
\draw[thick,gray] (25, {5  + ln(0.063428)}) -- (29,{5 + ln(0.05004)}) ;
\draw[thick,gray] (29, {5  + ln(0.05004)}) -- (30,{5 + ln(0.04721)}) ;
\node at (27.1,1.8) {k=4,j=3} ;

\draw[thick,violet] (1, {5  + ln(0.23315)}) -- (2,{5 + ln(0.20664)}) ;
\draw[thick,violet] (2, {5  + ln(0.2066)}) -- (3,{5 + ln(0.1820)}) ;
\draw[thick,violet] (3, {5  + ln(0.1820)}) -- (5,{5 + ln(0.1372)}) ;
\draw[thick,violet] (5, {5  + ln(0.1372)}) -- (9,{5 + ln(0.07817)}) ;
\draw[thick,violet] (9, {5  + ln(0.07817)}) -- (14,{5 + ln(0.03896)}) ;
\draw[thick,violet] (14, {5  + ln(0.03896)}) -- (18,{5 + ln(0.02221)}) ;
\draw[thick,violet] (18, {5  + ln(0.02221)}) -- (23,{5 + ln(0.01119)}) ;
\draw[thick,violet] (23, {5  + ln(0.01119)}) -- (24,{5 + ln(0.0098)}) ;
\node at (24.1,0.2) {k=6,j=4} ;

\node at (15,-0.7) {Number of iterations} ;
\node at (-4.3, 1.6) {s.d. of iterate} ;
\node at (-4.9,1.28) {(log scale)} ;

\end{tikzpicture}

\caption{The geometric decrease in s.d. for some $U[0,1]$ models.  The lines pass through the actual simulated values, without being fitted.}
\label{Fig:sd1}
\end{figure}

One can imagine analogous scaling limits near $0$ and $1$ for pairs $(j,k)$ for which the iterates converge to $\delta_{0,1}$.

 \begin{figure}[ht]
 \begin{center}
\includegraphics[width=2.5in]{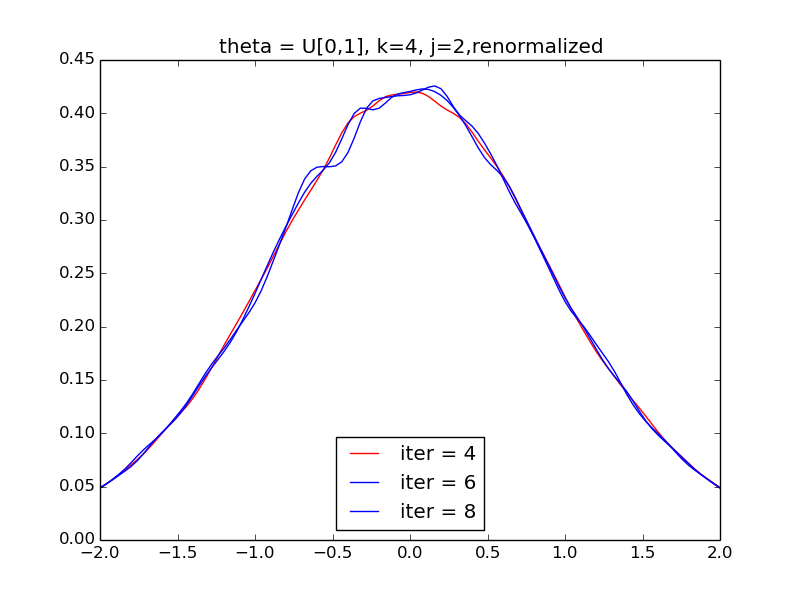}
\end{center}
\caption{Renormalized iterates from $U[0,1]$; \   $k = 4, j = 2$ strongly suggest a scaling limit.}
\label{Fig:U5}
\end{figure}

%\newpage

 \subsection{Non-uniform initial distributions.}
 \label{sec:non-u}
Some aspects of the observed behavior above might be consequences of symmetry (about $1/2$) of the initial distribution $U[0,1]$.
So here we study non-symmetric initial distributions.

First consider the slightly ``tilted" distribution $T[0,1]$ with density $f(u) = \frac{1}{2} + u$.
We find that the behavior remains qualitatively unchanged; compare Figure \ref{Fig:T1} below with Figures \ref{Fig:U1} and \ref{Fig:U2} above. 
Convergence to $\delta_{1/2}$ is replaced by convergence to some $\delta_s$ where $s$ depends on $(j,k)$.  
Convergence to $\delta_{0,1}$ is unchanged: Figure \ref{Fig:T1} shows that asymmetry persists through at least 5 iterations, but the probability masses 
near each endpoint tend to $1/2$ in the limit, consistent with the result that (for this case $k = 4$) the only invariant distribution supported on $\{0,1\}$ is the uniform distribution. 
 Figure \ref{Fig:T2} shows similar behavior  for the ``more tilted" distribution $TT[0,1]$ with density  $f(u) = 2u$.

 \begin{figure}[ht]
\includegraphics[width=2.3in]{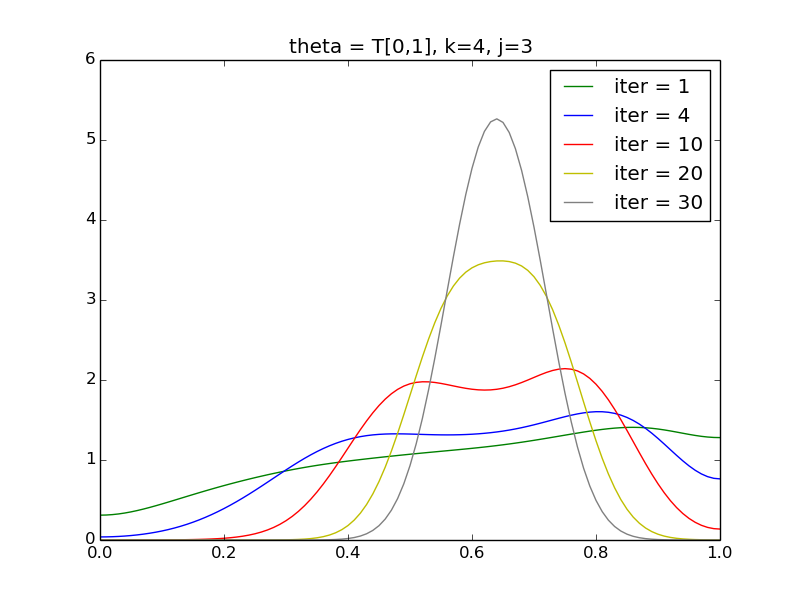}   \hspace*{-0.2in}
\includegraphics[width=2.3in]{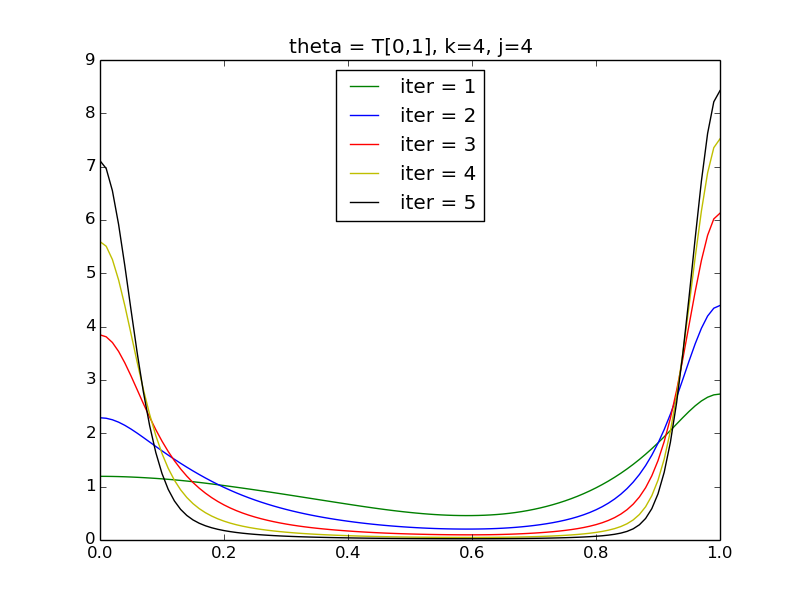} 
\caption{Iterates on the unit interval from a slightly tilted initial distribution, $k = 4, j = 3,4.$}
\label{Fig:T1}
\end{figure}

 \begin{figure}[ht]
\includegraphics[width=2.3in]{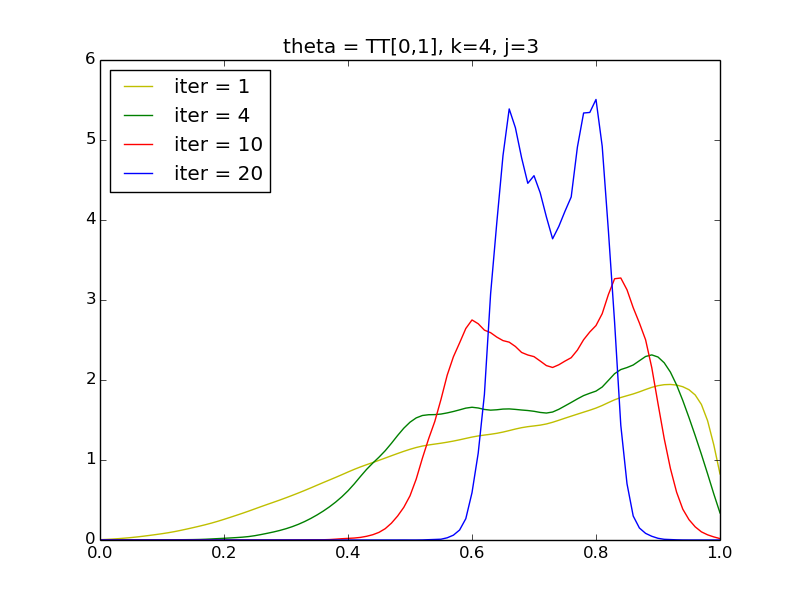}   \hspace*{-0.2in}
\includegraphics[width=2.3in]{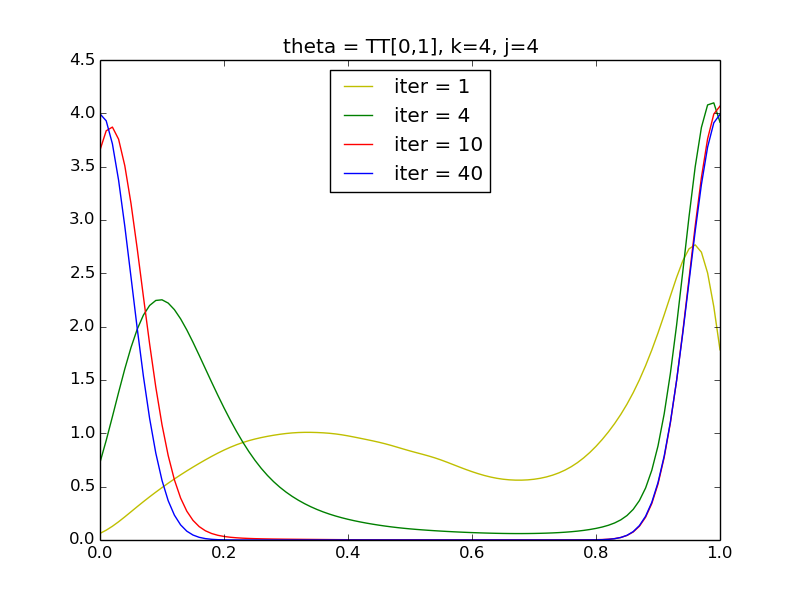} 
\caption{Iterates on the unit interval from a more tilted initial distribution, $k = 4, j = 3,4.$}
\label{Fig:T2}
\end{figure}

%\newpage
\subsection{Rigorous results for $S = [0,1]$}
\label{sec:rig}
The previous sections  stated ``results" about $\pi_{j,k}$ which we confidently believe via simulations, but without proofs.  
At a rigorous level, the key open questions are
\begin{itemize}
\item Does there exist (for any $(j,k)$) any invariant distribution with full support?
\item Does there exist (for any $(j,k)$) any distribution other than $\delta_s$ or $\delta_{0,1}$  that occurs as a limit of iterates from some initial distribution with full support?
\end{itemize}
We suspect the answer to each is ``no".
Of course, ``no" to the second question would imply ``no" to the first question.

A natural way to start a study is to write down the identity 
for an invariant density $f_{j,k}(x)$.
But because we are seeking to prove non-existence,
we seek to prove there is no solution.
Our partial results involve considering the density at endpoints; 
the proof in  \cite{mappings_short} of Theorem  \ref{P:0122} was constructed by first considering the density case and then re-writing the argument for general distributions.

Let us discuss in detail only the simplest case, that is $k = 2$. 
Recall from Theorem \ref{P:k=2} that the invariant distributions for $\pi_{2,2}$ and $\pi_{1,2}$ are the same.
%Recall that distributions of the form $\delta_{s}$ or $\delta_{s_1,s_2}$ are always invariant.

\begin{Theorem}[ \cite{mappings_short} Theorem 4]:
\label{P:0122}
 {\em There are no $\pi_{1,2}$ or $\pi_{2,2}$-invariant distributions on 
 the interval $[0,1]$ other than the omnipresent ones.}
\end{Theorem}
 The proof consists of analyzing the equation for an invariant density $f(x)$  to obtain a contradiction at $f(0)$.
A more elaborate argument (omitted) 
extends the result to some other values of $(j,k)$, as follows.

 {\em There is no $\pi_{j,k}$-invariant density on $[0,1]$ if $j = k$ or if
 \[
 \begin{pmatrix}
 &k& \\ j-1 & 1 & k-j
 \end{pmatrix}
 \ \frac{(j-1)^{j-1} (k-j)^{k-j}}{(k-1)^{k-1} } \le 2
 . \]
 }

 %See xxx Shi Feng notes.
%We presume the argument can be extended to general distributions to prove the analog of Theorem \ref{P:0122}, but have not considered details.

\subsection{Discrete points within $[0,1]$}
Somewhat relevant to the open problems at the start of section \ref{sec:rig} is the following discrete example.
\begin{Proposition}
The uniform distribution on the 4 points ${0, 0.4, 0.6, 1}$ is invariant for $\pi_{3,4}$.
\end{Proposition}
This is true by calculation: the transition matrix is 
\[
\frac{1}{256} \ 
\begin{pmatrix}
 13 &67& 109& 67\\
109& 13& 67 &67 \\
67& 67 &13 &109 \\
67& 109 &67 &13
\end{pmatrix}
\]
which is doubly-stochastic (like a variant of a Latin square).
But it is not clear if any analog holds for larger point sets.

 \section{The circle}
 \label{sec:circle}
 Here we take $S$ to be the unit circle $\CC = [0,1)$.

 \subsection{The uniform distribution}
 On the circle, the uniform distribution $U[\CC]$ is invariant, by symmetry.
 However, when we do Monte Carlo simulations of iterates, the symmetry breaks, as shown in  Figure \ref{Fig:C11}.
 We interpret this as implying that the uniform distribution is an {\em unstable} fixed point.
 
 In detail, the figure shows two realizations of each of the four cases ($j = 1,2,3,4$) with $k=4$.
 For each $j$, different realizations are qualitatively similar.
 The symmetry is broken  by the emergence of ``waves" 
with some number  $\nu = \nu(j,k)$ of peaks evenly spread around the circle.
The time (number of iterations) of emergence is similar in different realizations. 
In simulations  one wave eventually one takes over, as shown in Figure \ref{Fig:C12}.

Implications for the true iterative process are elusive.
Recall that by symmetry (section \ref{sec:symmetry}) the uniform distribution on any number
$\nu$ of points evenly spread around the circle (say $\theta_\nu$)  is invariant.
If a non-uniform initial distribution has some symmetry property, such as invariant under a $1/\nu$'th fraction of one rotation\footnote{For the maximum such $\nu$.},  then 
by the general ``preservation of symmetry" result (Lemma \ref{L:sym}) 
 that symmetry property  persists in the iterations and hence in any limit.
 So it is natural to conjecture that the iterative process started from such a distribution will converge to $\theta_\nu$.
However, the significance for the iterative process of the numbers $\nu(j,k)$ observed in initial iterates in simulations is elusive.

 \begin{figure}[p]
\includegraphics[width=2.3in]{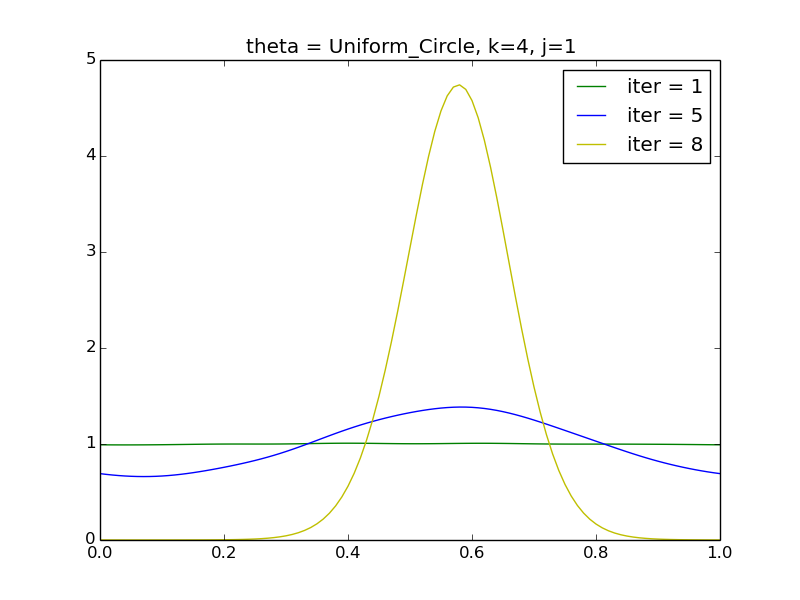}   \hspace*{-0.2in}
\includegraphics[width=2.3in]{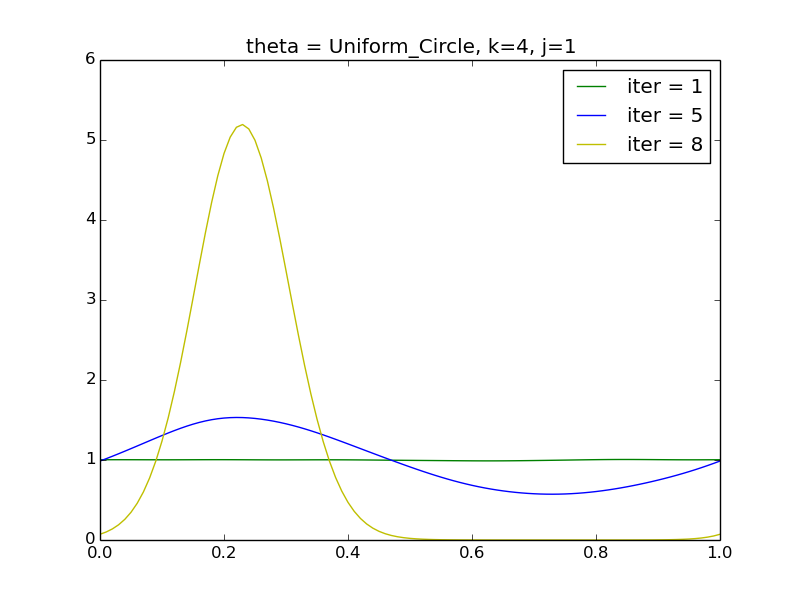} 

\includegraphics[width=2.3in]{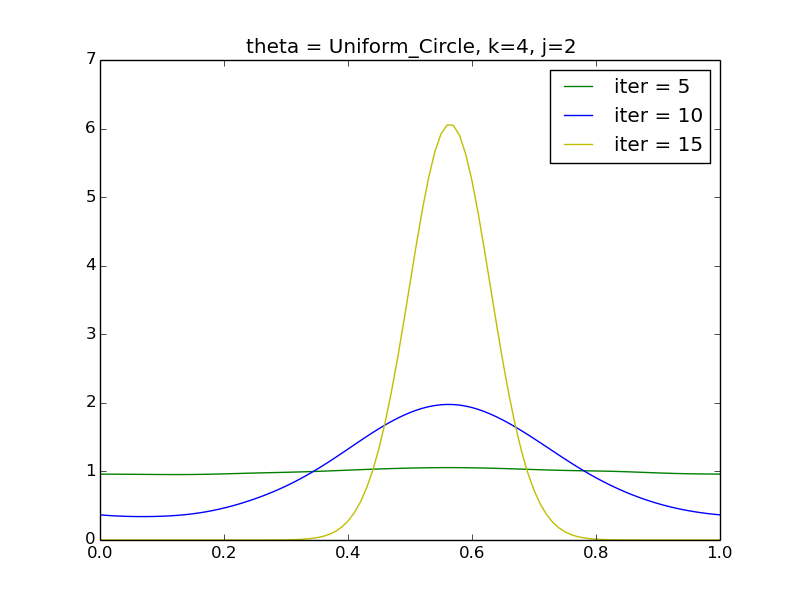}   \hspace*{-0.2in}
\includegraphics[width=2.3in]{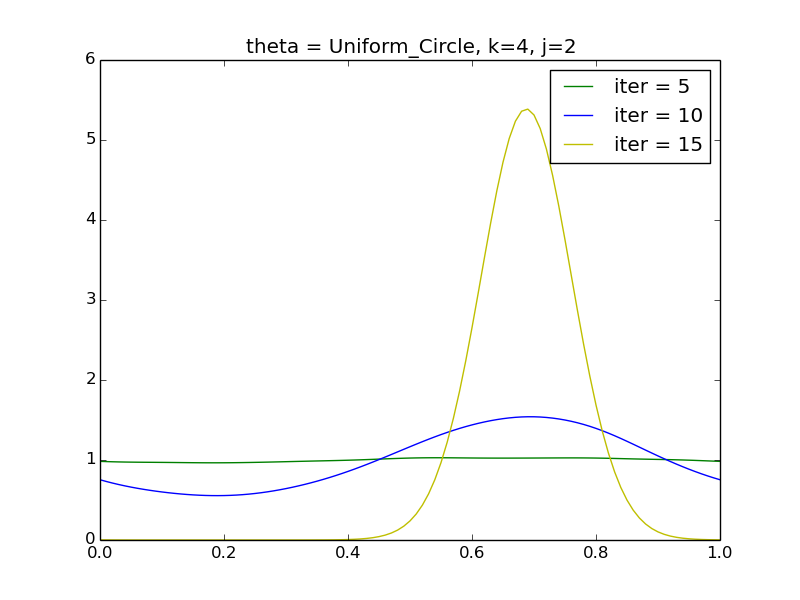} 

\includegraphics[width=2.3in]{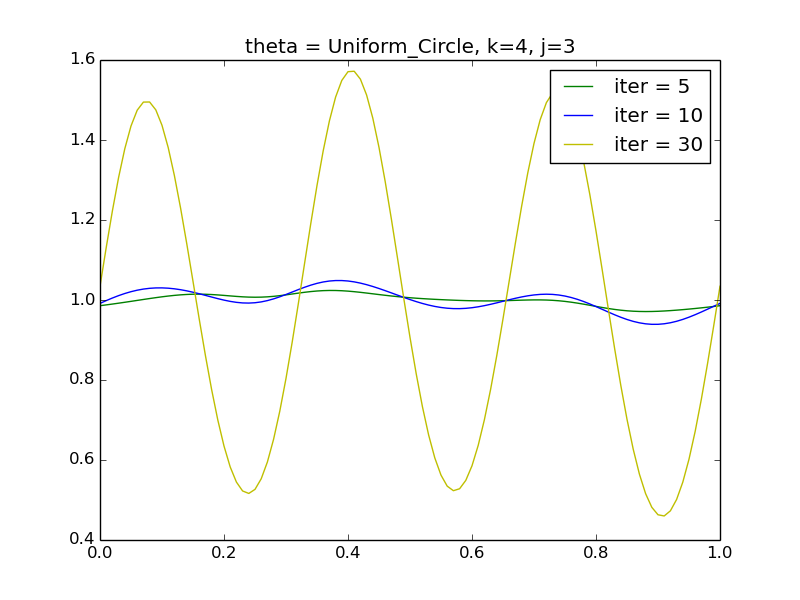}   \hspace*{-0.2in}
\includegraphics[width=2.3in]{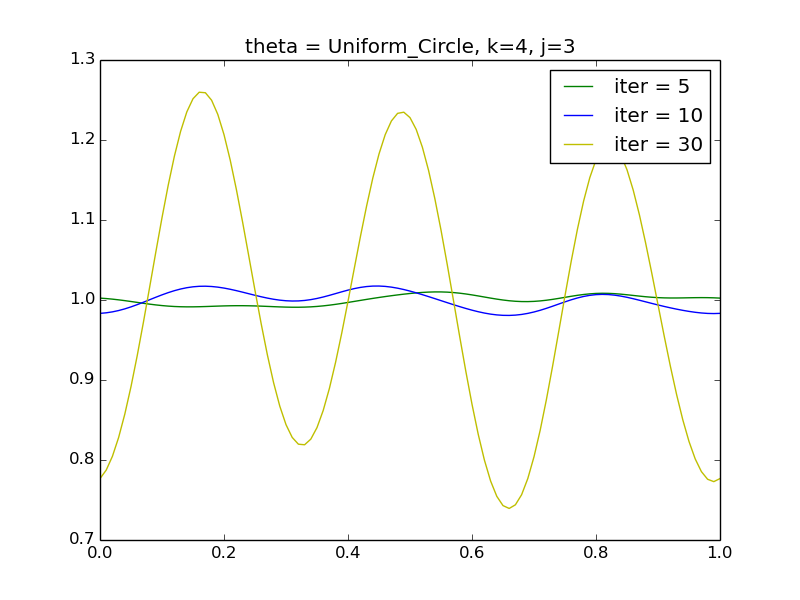} 

\includegraphics[width=2.3in]{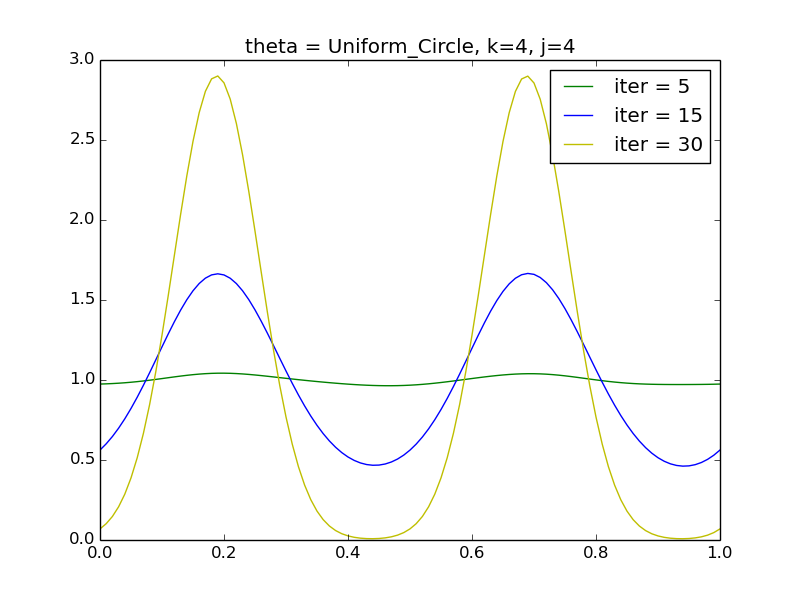}   \hspace*{-0.2in}
\includegraphics[width=2.3in]{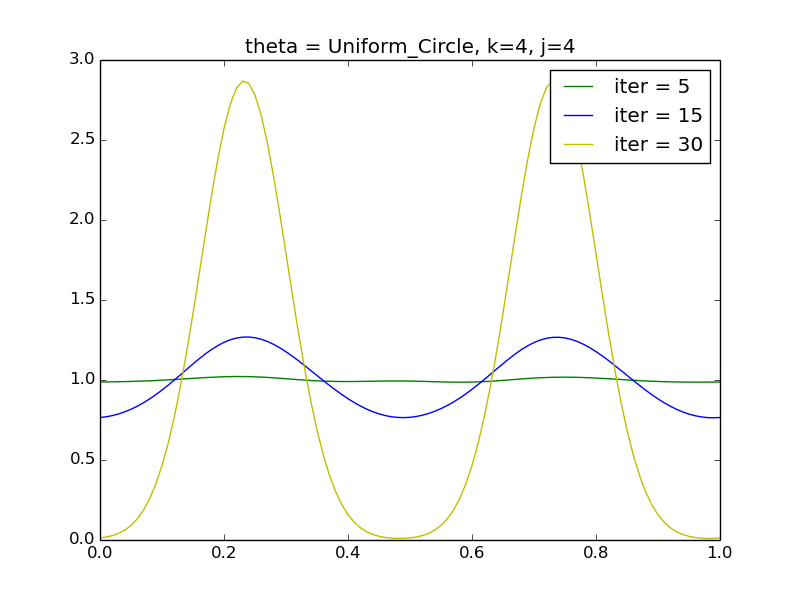} 

\caption{Iterates for $U[\CC];  \ k = 4, j = 1, 2, 3, 4$}
\label{Fig:C11}
\end{figure}

 \begin{figure}[ht]
\includegraphics[width=2.3in]{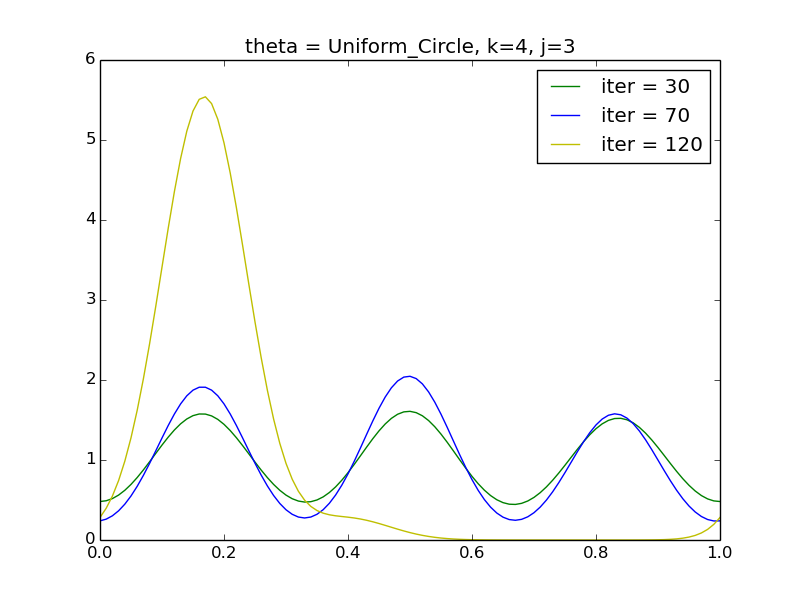}   \hspace*{-0.2in}
\includegraphics[width=2.3in]{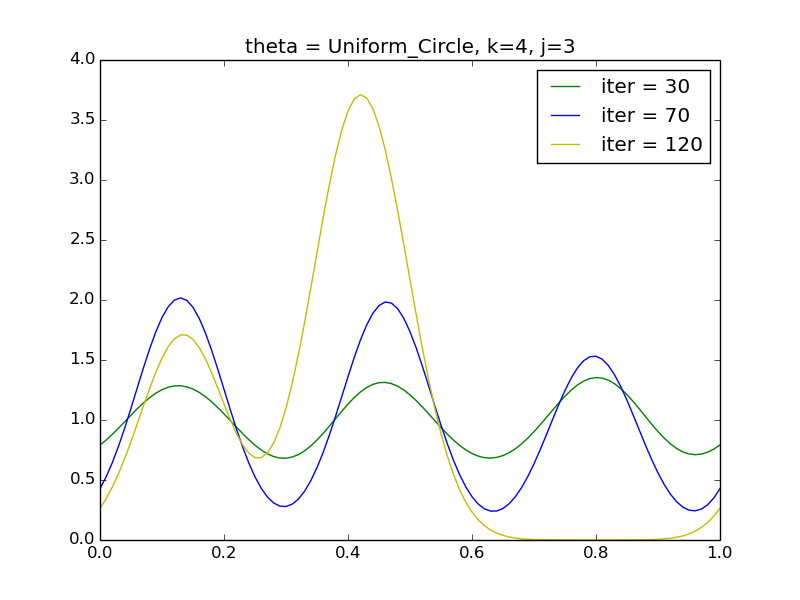} 
\caption{Iterates for $U[\CC]$, long-term: $k = 4, j = 3.$}
\label{Fig:C12}
\end{figure}

\subsection{Non-uniform initial distributions}
We have only done brief investigations of iterates on $\CC$ from non-symmetric starting distributions.
It appears that behavior here corresponds to  the case of the unit interval, in that
limits are either $\delta_s$ or $\delta_{s_1,s_2}$ where $s_1$ and $s_2$ are opposite points of $\CC$.
Figure \ref{Fig:CT12} shows the case of the  discontinuous initial density 
\begin{equation}
f(t) = t + 0.5, 0 < t < 1.
\label{disc-density}
\end{equation}

 \begin{figure}[ht]
\includegraphics[width=2.3in]{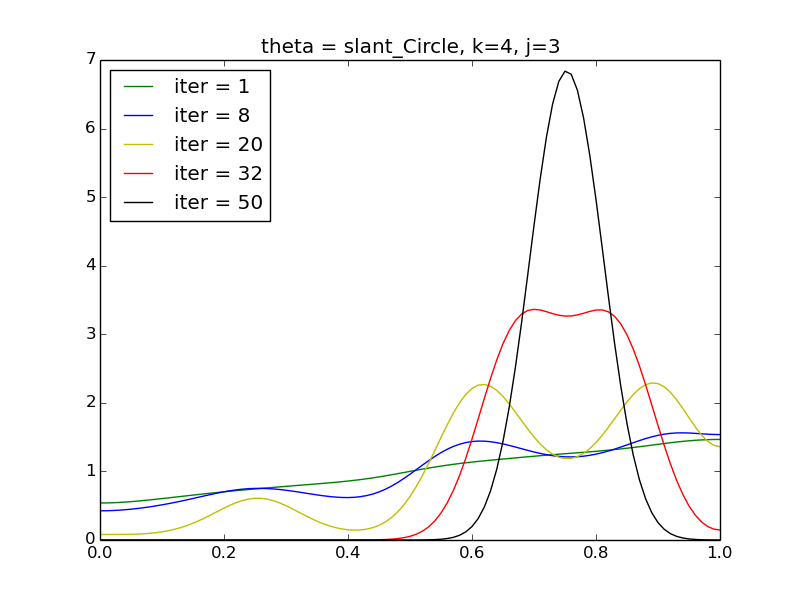}   \hspace*{-0.2in}
\includegraphics[width=2.3in]{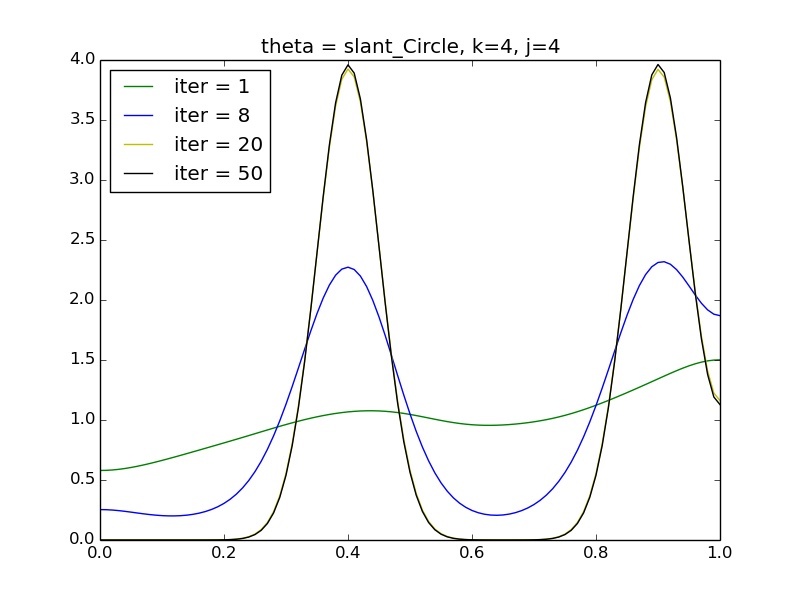} 
\caption{Iterates on $\CC$ from initial density (\ref{disc-density}), $k = 4, j = 3, 4.$}
\label{Fig:CT12}
\end{figure}

\section{A high-dimensional example}
\label{sec:highdim}
We wish to study a high-dimensional example without any special symmetry.
Here is our rather arbitrary choice.

Take $S = [0,1]^{10}$ with elements $\bx = (x_i)$ and with non-Euclidean metric
\begin{equation}
d(\bx,\by) = \sum_i \beta^i |x_i - y_i|  
\label{10-dim}
\end{equation}
for a parameter $0 < \beta < 1$.
And take initial distributions with density of the form
\[ f(\bx) \propto  \exp( - \alpha \sum_i x_i^2)    \]
for a parameter $\alpha > 0$.

Figure \ref{Fig:10-dim} shows simulations with $k = 2$.
What is shown is the distribution of ``distance to $\mathbf{0}$" after varying numbers of iterations.
The 8 panels show the 8 combinations of 
\[\alpha = 0.1 \mbox{ or } 1.0; \ \beta = 0.7 \mbox{ or } 0.9; \ j = 1 \mbox{ or }  2 .\]
The qualitative behavior is similar to previous examples: for $j = 1$ there is apparent convergence to some $\delta_{\bx}$ and 
for $j = 2$ we suspect there is convergence to $\delta_{\mathbf{0}, \mathbf{1}}$.

%\newpage

 \begin{figure}[p]
\includegraphics[width=2.3in]{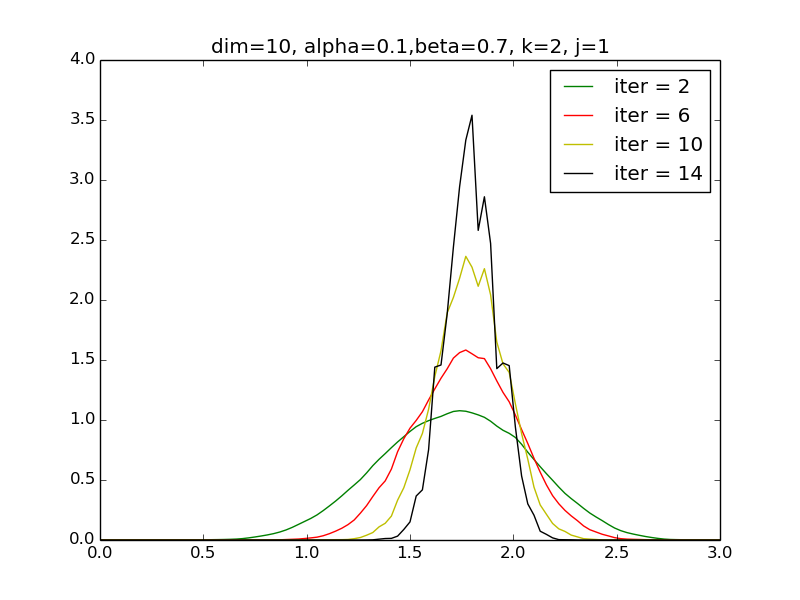}   \hspace*{-0.2in}
\includegraphics[width=2.3in]{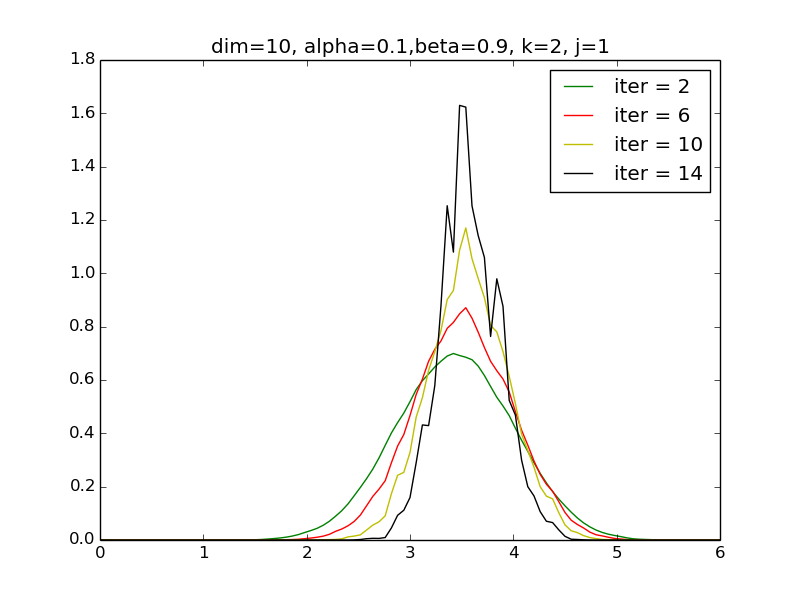} 

\includegraphics[width=2.3in]{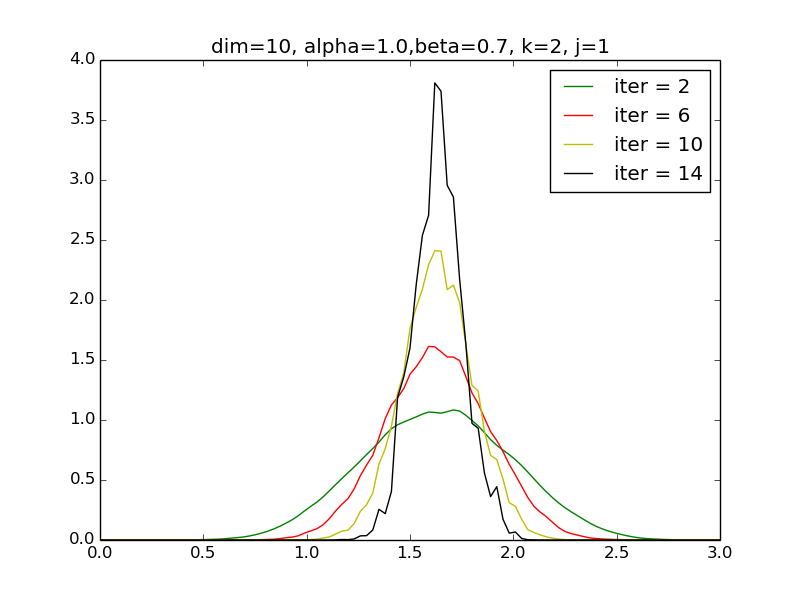}   \hspace*{-0.2in}
\includegraphics[width=2.3in]{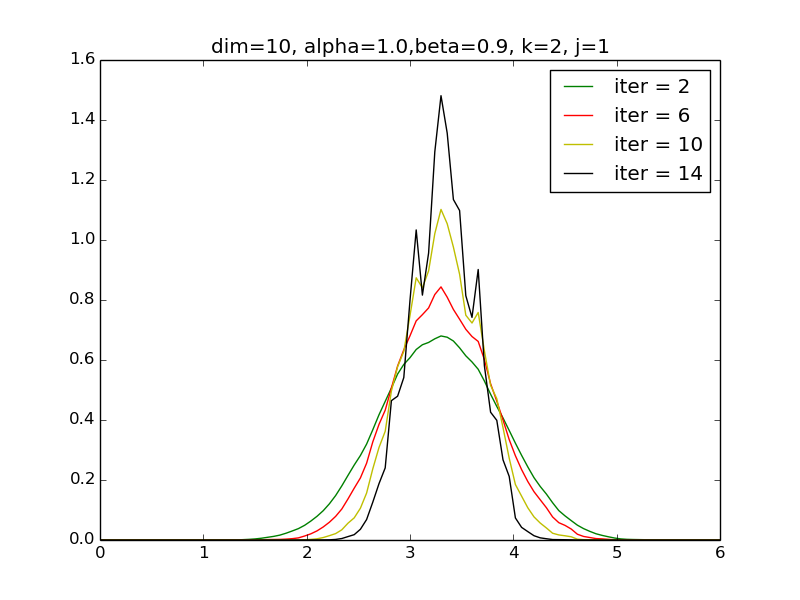} 

\includegraphics[width=2.3in]{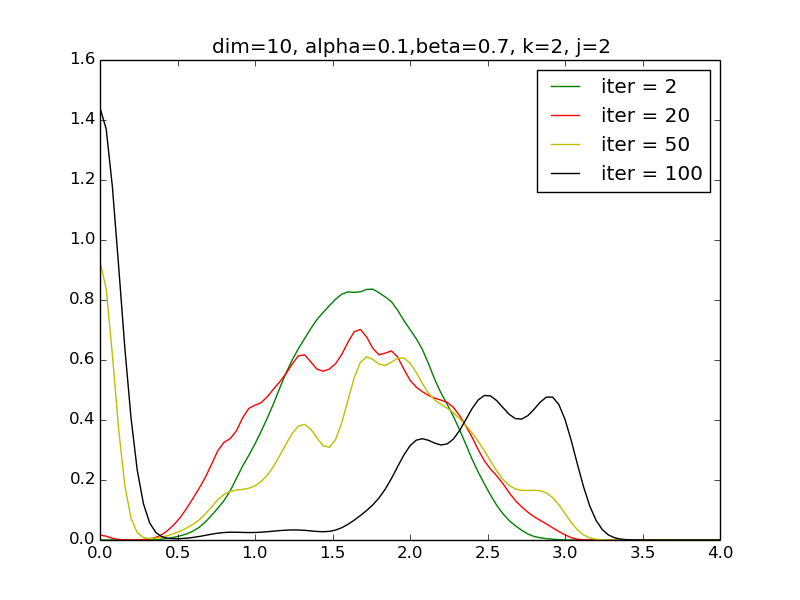}   \hspace*{-0.2in}
\includegraphics[width=2.3in]{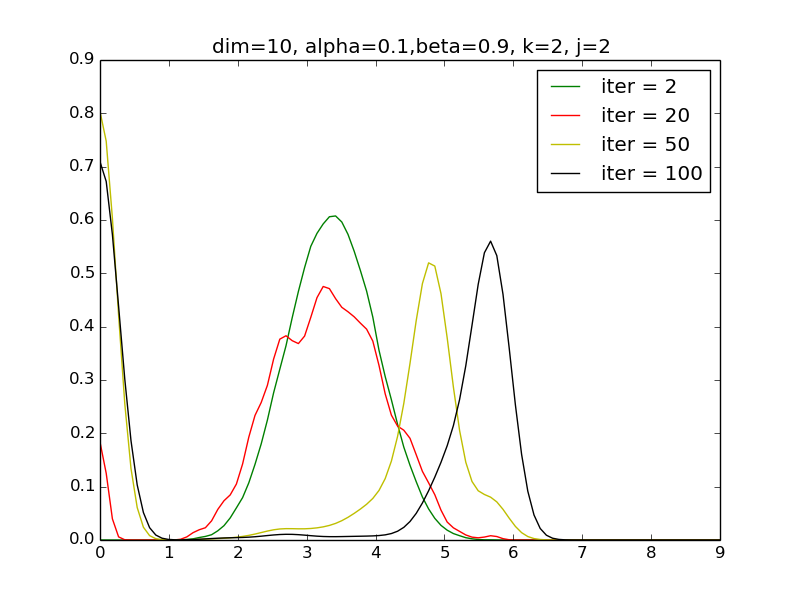} 

\includegraphics[width=2.3in]{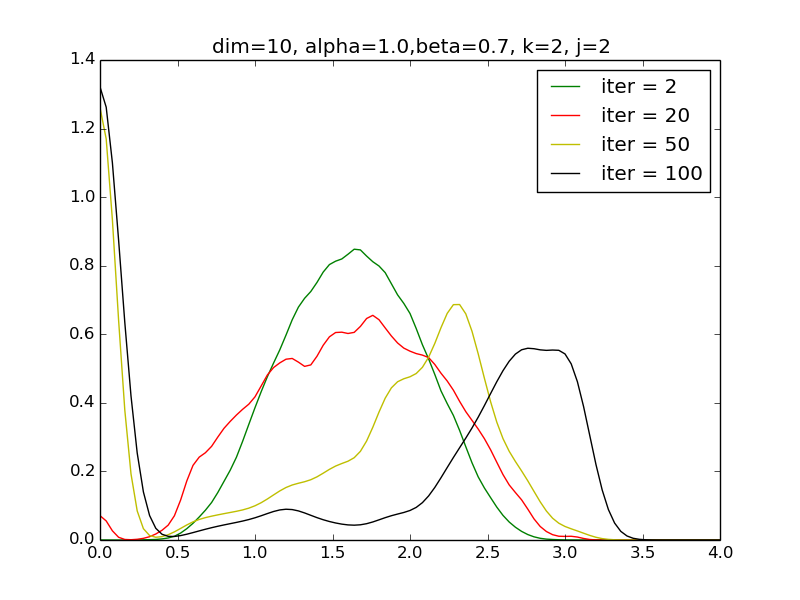}   \hspace*{-0.2in}
\includegraphics[width=2.3in]{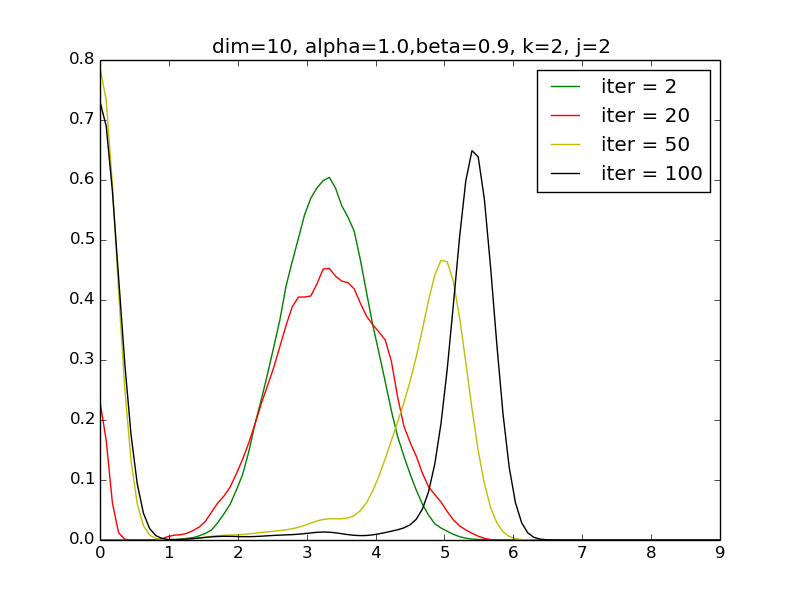} 
\caption{The 10-dimensional example (\ref{10-dim}); $k = 2.$}
\label{Fig:10-dim}
\end{figure}

\section{Final remarks}
\label{sec:34}

{\bf 1:  The $j \approx 3k/4$ heuristic.}
The heuristics  in section \ref{sec:U01} in the context of the unit interval can be repeated more generally.
Suppose $k$ is large, and suppose $S$ and $(j,k)$ are such that there is an initial tendency for the iterate distributions to concentrate near two distinct regions.
Then about half of the $k$ samples will be in each region, so for $j > k/2$
the chain will switch back and forth between the regions.
And the average distance between regions in successive iterations
will tend to increase or decrease according to whether the chosen point at each step of the chain is farther (from the other region) or nearer than the average, that is
according to whether $j > 3k/4$ or $j < 3k/4$.

{\bf 2.} 
The only rigorous analyses of continuous $S$ that we have completed involve explicit calculations.
To give rigorous proofs of, for instance, convergence of iterates to some $\delta_s$, 
it seems natural to investigate some kind of monotonicity property for concentration of distributions,
but we have been unable to implement that idea.

{\bf 3.}
This article studies a certain process which can be defined on any compact metric space $S$, parametrized by any $\theta \in \PP(S)$.
Are there other interesting such processes?
As one example, a certain such {\em coverage process} is studied in  \cite{me157}: ``seeds" arrive as a Poisson process in time at i.i.d. ($\theta$) locations and form the center of growing balls.

%\bibliographystyle{plain}
%\bibliography{../MM}

\end{document}